\documentclass[12pt, a4, reqno]{amsart}
\usepackage[a4paper, left=28mm, right=28mm, top=28mm, bottom=34mm]{geometry}
\usepackage[all]{xy}
\usepackage{amsmath}
\usepackage{amscd}
\usepackage{amssymb}
\usepackage{amsthm}
\usepackage{comment}

\usepackage{mathrsfs}
\theoremstyle{definition}
\newtheorem{thm}{Theorem}[section]
\newtheorem{lem}[thm]{Lemma}
\newtheorem{cor}[thm]{Corollary}
\newtheorem{prop}[thm]{Proposition}
\newtheorem{conj}[thm]{Conjecture}
\theoremstyle{definition}

\newtheorem{rem}[thm]{Remark}

\newtheorem{definition}[thm]{Definition}

\begin{document}

\title[Chow motives of genus one fibrations]{Chow motives of genus one fibrations}

\author{(*) Daiki Kawabe}

\date{March 9, 2024.}
\thanks{\textit{(*) Email address}: daiki.kawabe.math@gmail.com, \ \ \textit{ORCID}: 0000-0003-2766-7115\\
\textit{Conflict of interest:} This work is supported by the JSPS KAKENHI Grant Number 18H03667 and Tohoku University Global Hagi Scholarship.}

\maketitle

\begin{abstract}
Let $f: X \rightarrow C$ be a genus 1 fibration from a smooth projective surface, 
i.e. its generic fiber is a regular genus 1 curve.
Let $j: J \rightarrow C$ be the Jacobian fibration of $f$. 
In this paper, we prove that the Chow motives of $X$ and $J$ are isomorphic.
As an application, combined with our concomitant work on motives of quasi-elliptic fibrations, we prove Kimura finite-dimensionality for smooth projective surfaces not of general type with geometric genus 0.
This generalizes Bloch-Kas-Lieberman's result to arbitrary characteristic.
\end{abstract}

\section{Introduction}

\subsection{Motivation}
Let $k$ be an algebraically closed field of arbitrary characteristic.
Let $f: X \rightarrow C$ be a fibration from a smooth projective surface over $k$ to a smooth projective curve. 
It is a proper, surjective, $k$-morphism such that $f_{*}\mathcal{O}_{X} \cong \mathcal{O}_{C}$.\\
Let $\eta$ be the generic point of $C$ and $X_{\eta}$ the generic fiber of $f$.
We study the following:
\begin{enumerate}
\item $f$ is a \textit{genus $1$ fibration} if $X_{\eta}$ is a \textit{regular genus} $1$ curve.
More precisely,
\[ \text{$X_{\eta}$ is a \textit{regular}, projective, geometrically-integral, curve with arithmetic genus $1$.} 
\] 
\item A genus $1$ fibration $f$ is \textit{elliptic}  if $X_{\eta}$ is \textit{smooth}, i.e.
\textit{geometrically-regular}.
\item A genus $1$ fibration $f$ is \textit{quasi-elliptic} if $X_{\eta}$ is \textit{not smooth}.
\end{enumerate}
Note that quasi-elliptic surfaces exist only in characteristic $2$ or $3$ (Proposition \ref{non-smooth genus one curve}).
We say that a genus 1 fibration $f : X \rightarrow C$ is \textit{minimal} if every birational morphism 
$f : X \rightarrow X'$ over $C$ onto a smooth projective surface $X'$ is an isomorphism.\\
\indent From now on, we let $f : X \rightarrow C$ be a genus 1 fibration and assume $f$ is minimal if not stated otherwise.
Then $X_{\eta}$ does not necessarily have a $\eta$-rational point, 
hence $f$ may have multiple fibers.
To remedy this problem, we consider the Jacobian fibration $j : J \rightarrow C$ of $f$ (\cite{Cossec and Dolgachev}). Its generic fiber $J_{\eta}$ is the regular compactification of the Jacobian variety of $X_{\eta}$.
Then, $J$ is a smooth projective surface, and $j$ is a genus 1 fibration that has no multiple fibers.
Clearly, it is easier to study $J$ than $X$.
It is known that there are relations between some invariants of $X$ and $J$, e.g. the equalities of the $i$-th Betti numbers $b_{i}(X) = b_{i}(J)$ and the Picard numbers $\rho(X) = \rho(J)$ (\cite{Cossec and Dolgachev}).
For a smooth projective surface $S$ over $k$, we denote by $h(S)$ the Chow motive of $S$ with $\mathbb{Q}$-coefficients, 
and by $T(S)$ the kernel of the Albanese map $a_{S} : \mathrm{CH}_{0}(S)_{\mathbb{Z}}^{0} \rightarrow \mathrm{Alb}_{S/k}(k)$.\\
\indent In 1976, Bloch-Kas-Lieberman proved the following relation between $X$ and $J$$:$
\begin{prop}(\cite[Prop.~4,\ p.138]{Bloch and Kas and Lieberman}).\label{J implies X}
Let $f : X \rightarrow C$ be an elliptic fibration over $\mathbb{C}$
and $j : J \rightarrow C$ the Jacobian fibration of $f$. If $T(J) = 0$, then $T(X) = 0$.
\end{prop} 
In 1992, K.~Coombes proved the following relation between $X$ and $J$$:$
\begin{prop} (\cite[Prop.~3.1,\ p.52]{Coombes}). \label{Coombes's result}
Let $k$ be an algebraically closed field.
Let $X$ be an Enriques surface  over $k$ with an elliptic fibration $f : X \rightarrow \mathbb{P}^{1}$,
and $j : J \rightarrow \mathbb{P}^{1}$ the Jacobian fibraton of $f$.
Then there is an isomorphism of Chow motives $h(X) \cong h(J)$.
\end{prop}
The author was inspired by Propositions \ref{J implies X} and \ref{Coombes's result}.
In this paper, we generalize Proposition \ref{Coombes's result} for genus $1$ fibrations defined over arbitrary algebraically closed fields.\\

\indent S.~Kimura and O'Sullivan independently introduced 
a notion of finite-dimensionality for Chow motives 
(\cite[Def.~3.7]{Kimura} and \cite[Ch.\ 12]{Andre}).
We call it Kimura-finiteness.
They conjectured that every Chow motive is Kimura-finite.
Moreover, for a surface $S$ over $\mathbb{C}$ with $p_{g} = 0$, Kimura proved the theorem that $h(S)$ is Kimura-finite if and only if 
the Albanese map $a_{S} : \mathrm{CH}_{0}(S)_{\mathbb{Z}}^{0} \rightarrow \mathrm{Alb}_{S/\mathbb{C}}(\mathbb{C})$ is an isomorphism (\cite[Coro.~7.7]{Kimura}).
By using his theorem, the result of Bloch-Kas-Lieberman 
\cite{Bloch and Kas and Lieberman} (see Theorem \ref{result of BKL})
is equivalent to the following$:$
\begin{thm} \label{motivic vertion 1 of BKL} (= Theorem \ref{motivic vertion 2 of BKL}).
Let $X$ be a smooth projective surface over $\mathbb{C}$.
Assume that $X$ has geometric genus $0$ and Kodaira dimension $< 2$. 
Then $h(X)$ is Kimura-finite.
\end{thm}
In this paper, we generalize Theorem \ref{motivic vertion 1 of BKL} to arbitrary characteristic.

\subsection{Main theorems of this paper.} \ \indent \\
\indent We prove two main theorems (Theorems \ref{Main 1} and \ref{Main 4}).
The first one is the following:
\begin{thm} (= Theorem \ref{motive of genus one}). \label{Main 1} 
Let $k$ be an arbitrary algebraically closed field.
Let $f : X \rightarrow C$ be a minimal genus $1$ fibration from a smooth projective surface over $k$,
and $j : J \rightarrow C$ the Jacobian fibration of $f$.
Then, there is an isomorphism
\[ h(X) \cong h(J) \]
in the category $\mathrm{CH}\mathcal{M}(k, \mathbb{Q})$ 
of Chow motives over $k$ with $\mathbb{Q}$-coefficients.
\end{thm}

Theorem \ref{Main 1} is a generalization of Proposition \ref{Coombes's result} to genus $1$ fibrations.\\
Here, we give a sketch of the proof of Theorem \ref{Main 1}.
Let us consider the Chow-K\"unneth decompositions of Murre \cite{Murre On the Motive} (see Proposition \ref{Picard motive}) of $h(X)$ and $h(J)$, respectively
\begin{align*}
h(X) &\cong \oplus_{i=0}^{4}h_{i}(X) \cong 1 \oplus h_{1}(X) \oplus  h_{2}^{alg}(X) \oplus t_{2}(X) \oplus h_{3}(X) \oplus 
(\mathbb{L} \otimes \mathbb{L}) \\
h(J) &\cong \oplus_{i=0}^{4}h_{i}(J) \cong 1 \oplus h_{1}(J) \oplus  h_{2}^{alg}(J) \oplus t_{2}(J) \oplus h_{3}(J) \oplus
(\mathbb{L} \otimes \mathbb{L}). 
\end{align*}
Here, $1$ is the unit motive, $\mathbb{L}$ is the Lefschetz motive, 
and $h_{2}^{alg}(-)$ (resp. $t_{2}(-)$) is the algebraic (resp. transcendental) part of $h_{2}(-)$ (\cite{Kahn and Murre and Pedrini}).
Thus, it suffices to prove 
\[ h_{i}(X) \cong  h_{i}(J) \ \ \ \text{for $1 \leq i \leq 3$}. \]
\indent First, assume $i = 1$ or $3$.
For a smooth projective variety $V$ over $k$,
we denote by $(\mathrm{Pic}_{V/k}^{0})_{red}$ (resp. $\mathrm{Alb}_{V/k}$)
the Picard (resp. Albanese) variety of $V$.
We prove the following key proposition$:$
\begin{prop} \label{Main 2} (= Proposition \ref{Picard isog}).
There are isogenies of abelian $k$-varieties
\[ (\mathrm{Pic}_{X/k}^{0})_{red} \sim_{isog} (\mathrm{Pic}_{J/k}^{0})_{red}, \ \ \ \mathrm{Alb}_{X/k} \sim_{isog} \mathrm{Alb}_{J/k}. \]
\end{prop}
Using Proposition \ref{Main 2}, we prove $h_{i}(X) \cong h_{i}(J)$ for $i = 1$ or $3$.\\
\indent Finally, assume $i = 2$.
By $\rho(X) = \rho(J)$,
we get $h_{2}^{alg}(X) \cong 
\mathbb{L}^{\oplus \rho(X)}  \cong h_{2}^{alg}(J)$.\\
Thus, it remains to prove $t_{2}(X) \cong t_{2}(J)$.
The outline of the proof is as follows.\\
If $f$ is quasi-elliptic, so is $j$.
Using the result of the author (\cite{Kawabe}), we get 
\[ t_{2}(X) = 0 = t_{2}(J).\]
Thus, it suffices to consider the case where $f$ is elliptic.
Let us consider the following functors between the category of Chow motives
\[
\mathrm{CH}\mathcal{M}(\eta, \mathbb{Q}) \overset{i}\longleftarrow \mathrm{CH}\mathcal{M}(C, \mathbb{Q}) 
\overset{F}\longrightarrow \mathrm{CH}\mathcal{M}(k, \mathbb{Q}).\]
Here $\mathrm{CH}\mathcal{M}(C, \mathbb{Q})$ is the category of Chow motives over $C$ 
(\cite[Def.~2.8]{Corti and Hanamura}).
In particular, $i$ is not fully-faithful. Then, we consider the following two-step process.
\begin{enumerate}
\item We prove an isomorphism of Chow motives between the generic fibers
\[ h(X_{\eta}) \cong h(J_{\eta}) \ \ \  \text{in} \ \ \ \mathrm{CH}\mathcal{M}(\eta, \mathbb{Q}). \]
\item We extend the isomorphism $h(X_{\eta}) \cong h(J_{\eta})$ to the isomorphism 
\[ t_{2}(X) \cong t_{2}(J) \ \ \ \text{in} \ \ \  \mathrm{CH}\mathcal{M}(k, \mathbb{Q}). \]
\end{enumerate}
Concerning (i), we prove
\begin{thm} (= Theorem \ref{motive of smooth genus one}). \label{Main 3}
Let $K$ be an arbitrary field.
Let $A$ be an abelian variety over $K$.
Let $T$ be a torsor over $K$ for $A$.
Then there is an isomorphism $h(T) \cong h(A)$ 
in the category $\mathrm{CH}\mathcal{M}(K, \mathbb{Q})$ of Chow motives over $K$ with 
$\mathbb{Q}$-coefficients.
\end{thm}

\indent The second main theorem of this paper is the following$:$
\begin{thm} (= Theorem\ \ref{characteristic $p$}). \label{Main 4}
Let $X$ be a smooth projective surface over an algebraically closed field $k$ of characteristic $p \geq 0$.
Assume that $X$ has geometric genus $0$ and Kodaira dimension $< 2$, 
that is, $p_{g} = 0$ and $\kappa < 2$.
Then $h(X)$ is Kimura-finite in the category 
$\mathrm{CH}\mathcal{M}(k, \mathbb{Q})$ of Chow motives over $k$ with $\mathbb{Q}$-coefficients.
\end{thm}

Theorem \ref{Main 4} is a generalization of Theorem \ref{motivic vertion 1 of BKL} to arbitrary characteristic.\\
The outline of the proof of Theorem \ref{Main 4} is as follows.
If $\kappa < 0$, the assertion is clear.
Assume $\kappa = 0$ or $1$.
Then $X$ has a genus $1$ fibration $f : X \rightarrow C$ by the classification of surfaces.
We take the Jacobian fibration $j : J \rightarrow C$ of $f$, and prove $h(J)$ is Kimura-finite.
By Theorem \ref{Main 1}, we have $h(X) \cong h(J)$, and see $h(X)$ is Kimura-finite.

\subsection{Organization} 
This paper is organized as follows.\\
\indent The main parts of this paper are Sections $8$ (Theorem \ref{Main 1}) and $11$ (Theorem \ref{Main 4}).\\
In Section $2$, we recall several basic objects in algebraic geometry.
In Section $3$, we collect and prove some facts about relative correspondences.
In Section $4$, we recall the theory of Chow motives, e.g.
Chow-K\"unneth decompositions and transcendental motives.
In Section $5$, to prove Theorem \ref{Main 3}, we prove several facts about torsors for commutative group varieties.
We treat (not necessarily smooth) genus $1$ curves.\\
\indent In Section $6$, we prove Theorem \ref{Main 3} ($h(T) \cong h(A)$)
which is crucial for the proof of $t_{2}(X) \cong t_{2}(J)$.
In Section $7$, we recall relations between some invariants of a genus $1$ fibration and the associated Jacobian fibration.
Moreover, we prove $h_{i}(X) \cong h_{i}(J)$ for $i = 1, 3$ and $h_{2}^{alg}(X) \cong h_{2}^{alg}(J)$.
This section is based on \cite{Cossec and Dolgachev} and \cite{Cossec and Dolgachev and Liedtke and Kondo}.\\
\indent In Section $8$, using the results of Sections $2$ - $7$, we prove Theorem \ref{Main 1} ($h(X) \cong h(J)$).
We extend $h(X_{\eta}) \cong h(J_{\eta})$ to $t_{2}(X) \cong t_{2}(J)$.
In Section $9$, we collect some properties of Kimura-finiteness.
Recall that any curve has Kimura-finiteness and Kimura-finiteness is stable under direct sums, tensor products, and quotients (Proposition \ref{Properties of Kimura-finiteness}).
In Section $10$, we quickly review the classification of surfaces.\\
\indent In Section $11$, using the results of Sections $8$ - $10$, we prove Theorem \ref{Main 4}.

\subsection{Conventions and Terminology.}
We fix a base field $k$.
By $k$-variety we mean a reduced separated $k$-scheme of finite type.
Unless otherwise stated, we assume irreducibility for any $k$-variety.
By $k$-curve (resp. $k$-surface) we mean a variety of dimension $1$ (resp. dimension $2$).\\
\indent Let $i \in \mathbb{Z}_{\geq 0}$. For a $k$-scheme $X$,
We denote by $\mathrm{CH}_{i}(X)$ (resp. $\mathrm{CH}^{i}(X)$)
the Chow group of $i$-dimensional (resp. $i$-codimensional)
cycles on $X$ modulo rational equivalence with $\mathbb{Q}$-coefficients,
and set $\mathrm{CH}(X) = \oplus_{i}\mathrm{CH}^{i}(X)$.
For an irreducible $k$-variety $X$ , we have $\mathrm{CH}_{i}(X) = \mathrm{CH}^{\mathrm{dim}(X)-i}(X)$.
Let $f : X \rightarrow Y$ be a morphism of $k$-schemes.
If $f$ is proper, then $f_{*} : \mathrm{CH}_{i}(X) \rightarrow \mathrm{CH}_{i}(Y)$ denote the proper-pushfoward.
If $f$ is flat of relative dimension $l$, then $f^{*} : \mathrm{CH}_{i}(X) \rightarrow \mathrm{CH}_{i+l}(Y)$ denote the flat-pullback.\\
\indent Let $\mathcal{V}(k)$ denote the category of smooth projective $k$-varieties.
For $X, Y \in \mathcal{V}(k)$, 
set $\mathrm{Corr}^{r}(X, Y) : = \oplus_{\alpha}\mathrm{CH}^{\mathrm{dim}(X_{\alpha}) + r}(X_{\alpha} \times Y)$
where $X = \sqcup X_{\alpha}$, with $X_{\alpha}$ equidimensional.\\
$\bullet$ For $X$ an irreducible variety over $k$, we use following notations$:$\\
\indent \ \ \  $k(X)$ $:$ the function field of $X$\\
\indent \ \ \  $X_{M} := X \times_{\mathrm{Spec}(k)} \mathrm{Spec}(M)$ for any extension $M$ of $k$\\
\indent \ \ \ $X(M) : = \mathrm{Hom}_{\mathrm{Sch}(k)}(\mathrm{Spec}(M), X)$ for an extension $M$ of $k$\\
$\bullet$ For $X$ a projective variety over $k$, we use following notations$:$\\
\indent \ \ \  $h^{i}(\mathcal{F}) = h^{i}(X, \mathcal{F}) = \mathrm{dim}_{k} \ H^{i}(X, \mathcal{F})$
for any coherent sheaf $\mathcal{F}$ on $X$\\
\indent \ \ \  $\chi(\mathcal{F}) : = \sum_{i}(-1)^{i}h^{i}(\mathcal{F})$ for any coherent sheaf $\mathcal{F}$ on $X$\\
\indent \ \ \  $p_{a}(X) : = (-1)^{\mathrm{dim}(X)}(\chi(\mathcal{O}_{X}) - 1)$ $:$ the arithmetic genus\\
\indent \ \ \  $q(X) : = h^{1}(X, \mathcal{O}_{X})$ $:$ the irregularity\\
\indent \ \ \  $\omega_{X}$ $:$ the dualizing sheaf of $X$\\
$\bullet$ For $X$ a smooth projective variety over $k$, we use following notations$:$\\
\indent \ \ \  $\omega_{X}$ $:$ the canonical sheaf of $X$\\
\indent \ \ \  $K_{X}$ $:$ the canonical divisor of $X$\\
\indent \ \ \  $P_{m}(X) : = h^{0}(X, \omega_{X}^{\otimes m})$ $:$ the $m$-genus,
for $m = 1, 2, \cdot \cdot \cdot $\\
\indent \ \ \  $p_{g}(X) : = P_{1}(X)$ $:$ the geometric genus\\
\indent \ \ \  $b_{i}(X) : = \mathrm{dim}_{\mathbb{Q}_{l}} H_{\text{\'{e}t}}^{i}(X, \mathbb{Q}_{l})$ $:$ the $i$-th Betti number for a prime number $l \neq \mathrm{char}(k)$\\
\indent \ \ \  $e(X) : = \sum_{i}(-1)^{i}b_{i}(X)$ $:$ the topological Euler characteristic\\
$\bullet$ For $S$ a smooth projecitve surface over $k$, we have $p_{g}(S) = h^{0}(S, \omega_{S}) =  h^{2}(S, \mathcal{O}_{S})$.\\
$\bullet$ For simplicity, we use following notations$:$\\
\indent \ \ \ $X \cong Y$ \ \ \ \ \ \ \ $X$ and $Y$ are isomorphic.\\
\indent \ \ \ $X \sim_{birat} Y$\ \ \ $X$ and $Y$ are birationally equivalent as varieties\\
\indent \ \ \ $A \sim_{isog} B$ \ \ \ $A$ and $B$ are isogenous as abelian varieties\\

\textbf{In the rest of this paper, let $k$ be an algebraically closed field of arbitrary characteristic 
if not stated otherwise, i.e. $k = \overline{k}$.}

\section{Picard schemes}
In this section, we recall several basic objects in algebraic geometry.\\
$\bullet$ For a scheme $X$, we denote by $\mathrm{Pic}(X)$ the \textit{Picard group} of $X$.
Its elements are isomorphism classes of invertible sheaves on $X$.
Then
\[ \mathrm{Pic}(X) \cong H^{1}_{Zar}(X, \mathcal{O}_{X}^{*}) \cong H^{1}_{\textit{\'{e}t}}(X, \mathbb{G}_{m}). \]
The second isomorphism uses Hilbert's Theorem 90 (\cite[III.~Prop.~4.9, p.124]{Milne E}).\\
$\bullet$ For a scheme $X$, we denote by $\mathrm{Div}(X)$ 
the group of \textit{Cartier divisors} on $X$.\\
$\bullet$ Let $S$ be a scheme and $f : X \rightarrow S$ an $S$-scheme.
The \textit{relative Picard functor} $\underline{\mathrm{Pic}}_{X/S}$ is  defined by 
\[ \underline{\mathrm{Pic}}_{X/S} : (\mathrm{Sch}/S)^{\mathrm{op}} \rightarrow (\mathrm{Sets}) \ \ ; \ \
                       T  \mapsto  \mathrm{Pic}(X_{T})/f_{T}^{*}\mathrm{Pic}(T) \]
where $f_{T} : X_{T} : = X \times_{S} T \rightarrow T$ is the second projection.
We denote by $\underline{\mathrm{Pic}}_{(X/S)(\textit{\'{e}t})}$ 
(resp. $\underline{\mathrm{Pic}}_{(X/S)(\textit{fppf})}$) its associated sheaf in the
\'{e}tale (resp. fppf) topology.
\begin{thm}
Let $f : X \rightarrow S$ be a projective flat morphism between Noetherian schemes
with geometrically integral fibers.
Then the functor $\underline{\mathrm{Pic}}_{(X/S)(\textit{\'{e}t})}$ is represented 
by a group $S$-scheme $\mathrm{Pic}_{X/S}$ which is separated, locally of finite type over $S$.
Moreover, $\mathrm{Pic}_{X/S}$ is a disjoint union of open quasi-projective subschemes over $S$.
\end{thm}
\begin{proof}
 See \cite[Thm.~3.1, p.148]{GrV} or \cite[Thm.~9.4.8,\ p.263]{Kleiman}.
\end{proof}
\noindent The group scheme $\mathrm{Pic}_{X/S}$ is called the \textit{Picard scheme} of $X$ over $S$.
\begin{thm} \label{picard k}
Let $f : X \rightarrow S$ be a proper morphism between schemes.
Assume $S$ is a spectrum of a field $k$, and $X$ is complete.
Then the functor $\underline{\mathrm{Pic}}_{(X/k)(\textit{fppf})}$ is represented 
by a group scheme $\mathrm{Pic}_{X/k}$ which is separated, locally of finite type over $k$.
Moreover $\mathrm{Pic}_{X/k}$ is a disjoint union of open quasi-projective subschemes. 
\end{thm}
\begin{proof}
See \cite[Thm.~9.4.18.3,\ p.276]{Kleiman}.
The heart is \cite[Thm.~2,\ p.42]{Murre On contravariant}.
\end{proof}

\noindent $\bullet$ Let $G$ be a commutative group scheme over a field $k$, i.e. separated of locally finite type over $k$. Then $G$ is separated.
We denote by $G^{0}$ the connected component of the identity of $G$.
Then $G^{0}$ is a commutative group $k$-scheme of finite type.
We denote by $G_{red}$ the reduced scheme associated with $G$.
\begin{thm} \label{abelian picard variety}
Let $X$ be a proper scheme over a field $k$.
\begin{enumerate}
\item $\mathrm{Pic}^{0}_{X/k}$ is a group scheme which is a separated of finite type over $k$.
\item If $X$ is geometrically normal, then $\mathrm{Pic}^{0}_{X/k}$ is proper.
So $(\mathrm{Pic}^{0}_{X/k})_{red}$ is an abelian variety, and is called the \textit{Picard variety} of $X$.
\item If $h^{2}(\mathcal{O}_{X}) = 0$, then $\mathrm{Pic}^{0}_{X/k}$ is 
smooth of dim $h^{1}(\mathcal{O}_{X})$, so $\mathrm{Pic}^{0}_{X/k} = (\mathrm{Pic}^{0}_{X/k})_{red}$.
If also $X$ is geometrically-normal, $\mathrm{Pic}^{0}_{X/k}$ is an abelian variety of dim $h^{1}(\mathcal{O}_{X})$.
\end{enumerate}
\end{thm}
\begin{proof}
(i) follows from Theorem \ref{picard k} and the above argument.\\
(ii) See \cite[Thm.~2.1 (ii) p.231, Coro.3.2. p.236]{GrVI} or \cite[Rem.~9.5.6,\ p.277]{Kleiman}.
(iii) The smoothness is \cite[Prop.~2.10 (ii), p.236]{GrVI} or 
\cite[Prop. 9.5.19, p.285]{Kleiman}.
The dimension is \cite[Prop.~2.10 (iii), p.236]{GrVI} or \cite[Coro.~9.5.13, p.283]{Kleiman}.
\\
\end{proof}

\noindent $\bullet$ Let $C$ be a proper scheme over a field $k$ of dimension 1, 
$C = \sqcup_{i=1}^{r}C_{i}$ an its irreducible decomposition,
and $m_{i}$ the multiplicity of $C_{i}$.
The \textit{total degree} map of $C$ is defined by
\[ \mathrm{deg} : \mathrm{Pic}(C) \rightarrow \mathbb{Z} \ ; \ \mathcal{L} \mapsto \chi(\mathcal{L}) - \chi(\mathcal{O}_{C}). \]
We set $\mathrm{Pic}^{0}(C) = \mathrm{Ker}(\mathrm{deg})$.
The \textit{Jacobian variety} of $C$ is defined by 
\[ \mathrm{Jac}(C) : = \mathrm{Pic}^{0}_{C/k}. \]
Since $h^{2}(\mathcal{O}_{C}) = 0$, 
$\mathrm{Jac}(C)$ is smooth of dimension $h^{1}(\mathcal{O}_{C})$ by Theorem \ref{abelian picard variety} (iii).\\
\noindent $\bullet$ 
Assume $k = \overline{k}$. Let $X \in \mathcal{V}(k)$.
Then $\mathrm{Pic}(X) = (\mathrm{Pic}_{X/k})_{red}(k).$ The group 
\[ \mathrm{NS}(X) : = (\mathrm{Pic}_{X/k})_{red}(k)/(\mathrm{Pic}^{0}_{X/k})_{red}(k) 
\ \ (\text{resp}. \ \ 
\mathrm{Num}(X) : = \mathrm{NS}(X)/\mathrm{Torsion} \ \ )  \]
is called the \textit{Neron-Severi group} of $X$ (resp. the \textit{Picard lattice} of $X$). Then
\begin{prop} \label{Neron severi} (\cite[Coro.~9.6.17,\ p.298]{Kleiman}). 
$\mathrm{NS}(X)$ is a finitely-generated abelian group.
Its rank is called the \textit{Picard number} of $X$ and is denoted by $\rho(X)$.
\end{prop}
\noindent $\bullet$ Let $X$ be a geometrically-integral variety over a field $k$  
having a point $p_{0} \in X(k)$.\\
By \cite{Serre}, there are an abelian $k$-variety $\mathrm{Alb}_{X/k}$ and 
a $k$-morphism $\mathrm{alb}_{X} : X \rightarrow \mathrm{Alb}_{X/k}$ such that$:$
\begin{enumerate}
\item $\mathrm{alb}_{X}(p_{0}) = 0;$
\item for every $k$-morphism $g : X \rightarrow A$ of $X$ into an abelian variety $A$,
there is an unique $k$-homomorphism $g_{*} : \mathrm{alb}_{X} \rightarrow A$ such that $
g = g_{*} \circ \mathrm{alb}_{X}$.
\end{enumerate}
We call $\mathrm{Alb}_{X/k}$ the \textit{Albanese variety} of $X$ and 
call $\mathrm{alb}_{X}$ the \textit{Albanese morphism} of $X$.
If $X$ is smooth projective, $\mathrm{Alb}_{X/k}$ is the dual abelian variety of $(\mathrm{Pic}^{0}_{X/k})_{red}$.\\
\noindent $\bullet$
Assume $k = \overline{k}$.  Let $X \in \mathcal{V}(k)$ be an integral variety.
Fix a point $p_{0} \in X(k)$.\\
Let $\mathrm{CH}_{0}(X)_{\mathbb{Z}}^{0}$ be the Chow group of $0$-cycles of degree $0$ on $X$ with $\mathbb{Z}$-coefficients.
Then, there is a surjective homomorphism
\[ a_{X} : \mathrm{CH}_{0}(X)_{\mathbb{Z}}^{0} \rightarrow \mathrm{Alb}_{X/k}(k) 
\ \ ; \ \ \sum_{i}n_{i}[p_{i}] \mapsto \sum_{i}n_{i}[\mathrm{alb}_{X}(p_{i})]. \]
This map is called the \textit{Albanese map} of $X$. 
Moreover, its kernel is called the \textit{Albanese Kernel} of $X$, and is denoted by $T(X)$.\\ 
\noindent $\bullet$ The \textit{cohomological Brauer group} of a scheme $X$ is defined by 
$\mathrm{Br}(X) : = H^{2}_{\textit{\'{e}t}}(X, \mathbb{G}_{m})$.
If $K$ is a field, we set $\mathrm{Br}(K) : = \mathrm{Br}(\mathrm{Spec}(K))$.
\begin{prop} \label{exact brauer}
Let $f : X \rightarrow S$ be a separated morphism of finite type between locally Noetherian schemes.
Assume $f_{*}\mathcal{O}_{X} \cong \mathcal{O}_{S}$ holds universally, i.e.
$f_{T*}\mathcal{O}_{X_{T}} \cong \mathcal{O}_{T}$ for any $S$-scheme $T$. 
Then, there is an exact sequence
\[ 0 \rightarrow \mathrm{Pic}(X_{T})/\mathrm{Pic}(T) \overset{\alpha}\rightarrow \underline{\mathrm{Pic}}_{(X/S)(\textit{\'{e}t})}(T) \overset{\delta}\rightarrow \mathrm{Br}(T)  \]
for any $S$-scheme $T$. The map $\alpha$ is bijective if $f_{T}$ has a section or if $\mathrm{Br}(T) = 0$.
\end{prop}
\begin{proof}
Let us consider the Leray spectral sequence 
$E_{2}^{p, q} = H^{p}(T, R^{q}f_{T*}\mathbb{G}_{m}) \Rightarrow H^{p+q}(X_{T}, \mathbb{G}_{m})$.
Then \textit{the exact sequence of terms of low degree} is$:$
\[0 \rightarrow H^{1}(T, f_{T*}\mathbb{G}_{m}) \rightarrow H^{1}(X_{T}, \mathbb{G}_{m}) \rightarrow  
H^{0}(T, R^{1}f_{T*}\mathbb{G}_{m}) \rightarrow H^{2}(T, f_{T*}\mathbb{G}_{m}) \rightarrow H^{2}(X_{T}, \mathbb{G}_{m}). \]
Since $f_{*}\mathcal{O}_{X} \cong \mathcal{O}_{S}$ holds universally, the above exact sequence becomes
\[ 0 \rightarrow \mathrm{Pic}(T) \rightarrow \mathrm{Pic}(X_{T}) \rightarrow 
\underline{\mathrm{Pic}}_{(X/S)(\textit{\'{e}t})}(T) \rightarrow 
\mathrm{Br}(T) \rightarrow \mathrm{Br}(X_{T}). \]
Thus the assertion follows.
\end{proof}

\section{Correspondences}
In this section, we recall and prove some facts about correspondences.
\subsection{Correspondences over a field} Let $k$ be a field and $X, Y \in \mathcal{V}(k)$.
\begin{definition}
A \textit{correspondence} from $X$ to $Y$ is an element of $\mathrm{CH}(X \times Y)$.\\
For simplicity, we write $\alpha \in \mathrm{CH}(X \times Y)$ as $\alpha : X \vdash Y$. 
\end{definition}
If $\alpha : X \vdash Y$, $\beta : Y \vdash Z$, the composition $\beta \circ \alpha : X \vdash Z$ is defined by 
\[ \beta \circ \alpha : = p_{XZ*}(p_{XY}^{*}(\alpha) \cdot p_{YZ}^{*}(\beta)). \] 
Here $p_{XY}$, $p_{YZ}$, $p_{XZ}$ denote the projections from $X \times Y \times Z$ to 
$X \times Y$, $Y \times Z$, $X \times Z$.\\
\indent For $\alpha : X \vdash Y$, define a homomorphism $\alpha_{*} : \mathrm{CH}(X) \rightarrow \mathrm{CH}(Y)$ by $\alpha_{*}(a) = p_{Y*}^{XY}(\alpha \cdot p_{X}^{XY*}(a))$,
and a homomorphism $\alpha^{*} : \mathrm{CH}(Y) \rightarrow \mathrm{CH}(X)$ by
$\alpha^{*}(b) = p_{X*}^{XY}(\alpha \cdot p_{Y}^{XY*}(b))$.
A correspondence $\alpha : X \vdash Y$ has a transpose ${}^{t}\alpha : Y \vdash X$ defined by 
${}^{t}\alpha = \tau_{*}(\alpha)$
where $\tau : X \times Y \rightarrow Y \times X$ reverses the factors.
For any morphism $f : X \rightarrow Y$, we denote by
\[ \Gamma_{f} : X \vdash Y \]
the graph of $f$.
If $f = \delta_{X} : X \hookrightarrow X \times X$ is the diagonal embedding, set $\Delta_{X} : = \Gamma_{f}$.

\begin{lem} 
$($Lieberman's lemma$)$.
Let $\alpha : X \vdash Y$ and $\beta : X' \vdash Y'$.
Let $f : X \rightarrow X'$ and $g : Y \rightarrow Y'$ be proper and flat morphisms. Then
\[  (f \times g)_{*}(\alpha) =  \Gamma_{g} \circ \alpha \circ {}^{t}\Gamma_{f}. \]
\end{lem}
\begin{proof}
It follows from \cite[Prop.~16.1.1]{Fulton}.
\end{proof}

For any $X, T \in \mathcal{V}(k)$, let $X(T) : = \mathrm{CH}(T \times X)$.
For $\phi : X \vdash Y$, we define
\[ \phi_{T} : X(T) \rightarrow Y(T) \ \ ; \ \  \alpha \mapsto \phi \circ \alpha. \]
\begin{thm}
$($Manin's identity principle$)$
Let $\phi, \psi : X \vdash Y$.
Then 
\[ (\mathrm{i}) \ \phi = \psi \ \ \ \ \Longleftrightarrow \ \ \ \
(\mathrm{ii}) \ \phi_{T} = \psi_{T} \ \text{for all $T \in \mathcal{V}(k)$} \ \ \ \ \Longleftrightarrow 
\ \ \ (\mathrm{iii}) \ \phi_{X} = \psi_{X}.\]
\end{thm}
\begin{proof}
(i) $\Rightarrow$ (ii) $\Rightarrow$ (iii) are trivial.
(iii) $\Rightarrow$ (i) follows from taking $\alpha = \Delta_{X}$.
\end{proof}

\begin{rem} \label{identity principal} 
The Manin identity principle induces
\[ \Gamma_{f} = \Gamma_{g}  \ \text{in} \  \mathrm{CH}(X \times Y) \ \  \Longleftrightarrow \ \  (\mathrm{id}_{T} \times f)_{*} =  (\mathrm{id}_{T} \times g)_{*} :  \mathrm{CH}(T \times X) \rightarrow \mathrm{CH}(T \times Y) \  \forall T \]
\end{rem}

\begin{prop} \label{Galois extension correspondence}
Let $X, Y \in \mathcal{V}(k)$.
Let $\pi : X \rightarrow Y$ be a finite morphism.
\begin{enumerate} 
\item Let $d$ be the degree of $\pi$.
Then $\Gamma_{\pi} \circ {}^{t}\Gamma_{\pi} = d \cdot \Delta_{Y}$ in $\mathrm{CH}(Y \times Y)$.
\item Let $G$ be a finite group which acts freely on $X$, and let $Y : = X/G$.
Then  ${}^{t}\Gamma_{\pi} \circ \Gamma_{\pi} = \sum_{\sigma \in G}\Gamma_{\sigma}$ in
$\mathrm{CH}(X \times X)$.
\end{enumerate}
\end{prop}
\begin{proof}
The proof of (i) is similar to (ii), so it suffices to prove (ii).
Let $\mathrm{CH}(X)^{G}$ be the $G$-invariant subgroup.
Then $\mathrm{CH}(Y) \cong \mathrm{CH}(X)^{G}$ by \cite[Exa.~1.7.6]{Fulton},
so $\pi^{*} \pi_{*} = \sum_{\sigma \in G}\sigma_{*}$ in $\mathrm{Aut}(\mathrm{CH}(X))$,
and hence 
${}^{t}\Gamma_{\pi} \circ \Gamma_{\pi} = \sum_{\sigma \in G}\Gamma_{\sigma}$
by Remark \ref{identity principal}.
\end{proof}
 
\subsection{Relative correspondences and base changes} \ \indent \\
\indent To extend $h(X_{\eta}) \cong h(J_{\eta})$ to $t_{2}(X) \cong t_{2}(J)$, we prove some facts about relative correspondences.
This subsection is based on \cite{Corti and Hanamura} and \cite[Ch.~8]{Murre and  Nagel and Peters}.\\
\indent Let $B$ be a quasi-projective variety over a field $k$.
Now, we explain some concepts.\\
$\bullet$ Let $\mathcal{V}(B)$ be the category whose objects are pairs $(X, f)$ with 
$X$ a smooth quasi-projective $k$-variety and $f : X \rightarrow B$ a projective morphism.\\
$\bullet$ A morphism from $(X, f)$ to $(Y, g)$ is a morphism $h : X \rightarrow Y$ such that $g \circ h = f$.\\
$\bullet$ Let $(X, f)$, $(Y, g) \in \mathcal{V}(B)$.
Assume that $Y$ is equidimensional. Set
\[ \mathrm{Corr}^{r}_{B}(X, Y) : = \mathrm{CH}_{\mathrm{dim}(Y) - r}(X \times_{B} Y) \ \ \ \text{and} \ \ \ \mathrm{Corr}_{B}(X, Y) : = \oplus_{r}\mathrm{Corr}^{r}_{B}(X, Y). \]
$\bullet$ There are defined for Cartesian squares
\[
\begin{CD}
X @>>> Y \\
@VVV @VVV\\
B @>{i}>> C.
\end{CD}
\]
If $i$ is regular embedding of codimension $d$, 
the upper map induces $i^{!} : \mathrm{CH}_{k}(Y) \rightarrow \mathrm{CH}_{k-d}(X)$.
We apply this construction using the following diagram with right-hand side Cartesian squares
\[
\begin{CD}
X \times_{B} Y  @<{p_{XZ}}<< X \times_{B} Y \times_{B} Z @>{\delta'}>>
(X \times_{B} Y) \times_{k} (Y \times_{B} Z) \\
&& @VVV  @VVV \\
& & Y @>{\delta}>>  Y \times_{k} Y.
\end{CD} 
\]
Since $Y$ is smooth, 
$\delta$ is a regular embedding and so the refined Gysin homomorphism $\delta^{!}$ is well-defined.
For $\Gamma_{1} \in \mathrm{Corr}^{r}_{B}(X, Y)$ and $\Gamma_{2} \in \mathrm{Corr}^{s}_{B}(Y, Z)$,
we define
\[ \Gamma_{2} \circ_{B} \Gamma_{1} : = (p_{XZ})_{*}((\delta')^{!}(\Gamma_{1} \times_{k} \Gamma_{2})) 
\in \mathrm{Corr}^{r+s}_{B}(X, Z). \]
\indent The following two lemmas say that the composition of relative correspondences is compatible with some base change.
\begin{lem} \label{correspondences and base change}
Let $X, Y, Z \in \mathcal{V}(B)$.
Let $i : U \hookrightarrow B$ be an open immersion.\\
Let $i_{XY} : X \times_{B} Y \times_{B} U \rightarrow X \times_{B} Y$ be the projection, 
and similarly for $i_{VW}$ and $i_{UW}$.\\
Let $\Gamma_{1} \in \mathrm{Corr}_{B}(X, Y)$ and $\Gamma_{2} \in \mathrm{Corr}_{B}(Y, Z)$.
Then in $\mathrm{Corr}_{U}(X, Z)$
\[ ((i_{YZ})^{*}\Gamma_{2}) \circ_{U} ((i_{XY})^{*}\Gamma_{1}) = 
(i_{XZ})^{*}(\Gamma_{2} \circ_{B} \Gamma_{1}). \]
\end{lem}
\begin{proof}
Since $i$ is an open immersion, there are the following Cartesian diagrams
\[
\begin{CD}
X \times_{B} Z \times_{B} U  @<{p'_{XZ}}<< X \times_{B} Y \times_{B} Z \times_{B} U @>{\delta''}>> 
(X \times_{B} Y \times_{B}  U) \times_{k} (Y \times_{B} Z \times_{B} U) \\
@VV{i_{UW}}V    @VV{q}V  @VV{p}V \\
X \times_{B} Z   @<{p_{XZ}}<<  X \times_{B} Y \times_{B} Z @>{\delta'}>> 
(X \times_{B} Y) \times_{k} (Y \times_{B} Z)\\
& & @VVV @VVV\\
& & Y @>{\delta}>> Y \times_{k} Y
\end{CD}
\]
where the morphisms $i_{XZ}$, $p$, $q$ are open immersion.
Then in $\mathrm{CH}(X \times_{B} Y \times_{B} Z \times_{B} U)$
\begin{equation} \label{euality for base change}
(\delta'')^{!}(i_{XY} \times_{k} i_{YZ})^{*}(\Gamma_{1} \times_{k} \Gamma_{2}) 
= (\delta'')^{!}p^{*}(\Gamma_{1} \times_{k} \Gamma_{2}) =
q^{*}(\delta')^{!}(\Gamma_{1} \times_{k} \Gamma_{2}) 
\end{equation}
Apply to $(p'_{XZ})_{*}$ to this formula to obtain$:$
\begin{align*}
((i_{YZ})^{*}\Gamma_{2}) \circ_{U} ((i_{XY})^{*}\Gamma_{1}) 
&= (p'_{XZ})_{*}(\delta'')^{!}((i_{YZ})^{*}\Gamma_{2} \times_{k} (i_{XY})^{*}\Gamma_{1})\\
&= (p'_{XZ})_{*}(\delta'')^{!}(i_{YZ} \times_{k} i_{XY})^{*}(\Gamma_{2} \times_{k} \Gamma_{1})\\
&= (p'_{XZ})_{*}q^{*}(\delta')^{!}(\Gamma_{1} \times_{k} \Gamma_{2}) &&\text{by (\ref{euality for base change})}\\
&= (i_{XZ})^{*}(p_{XZ})_{*}(\delta')^{!}(\Gamma_{1} \times_{k} \Gamma_{2}) && \text{by base change theorem} \\
&= (i_{XZ})^{*}(\Gamma_{1} \circ_{B} \Gamma_{2}).
\end{align*}
\end{proof}

\begin{lem} \label{correspondences and closed immersion}
(\cite[Lem.~8.1.6,\ p.108]{Murre and  Nagel and Peters}).
Let $X, Y, Z \in \mathcal{V}(B)$. \label{correspondences and canonical morphism}
Let $t : B \rightarrow B'$ be a $k$-morphism.
Let $j_{XY} : X \times_{B} Y \rightarrow X \times_{B'} Y$ be the canonical morphism, 
and similarly for $j_{XY}$ and $j_{XZ}$.
Let $\Gamma_{1} \in \mathrm{Corr}_{B}(X, Y)$, $\Gamma_{2} \in \mathrm{Corr}_{B}(Y, Z)$.
Then in $\mathrm{Corr}_{B'}(X, Z)$
\[ ((j_{YW})_{*}\Gamma_{2}) \circ_{B'} ((j_{XY})_{*}\Gamma_{1}) = 
(j_{XZ})_{*}(\Gamma_{2} \circ_{B} \Gamma_{1}). \]
\end{lem}
\begin{proof}
Let us consider the following commutative diagrams
\[
\begin{CD}
X \times_{B} Z   @<{p_{XZ}}<< X \times_{B} Y \times_{B} Z  @>{\delta''}>> 
(X \times_{B} Y) \times_{k} (Y \times_{B} Z) \\
@VV{j_{XZ}}V    @VV{q}V  @VV{p}V \\
X \times_{B'} Z   @<{p'_{XZ}}<<  X \times_{B'} Y \times_{B'} Z @>{\delta'}>> 
(X \times_{B'} Y) \times_{k} (Y \times_{B'} Z)\\
& & @VVV @VVV\\
& & Y @>{\delta}>> Y \times_{k} Y.
\end{CD}
\]
The remainder is similar to the proof of Lemma \ref{correspondences and base change}.
\end{proof}

\section{Chow motives}
In this section, we review the theory of Chow motives.
\subsection{The category of Chow motives} 
Let $\mathrm{CH}\mathcal{M}(k) = \mathrm{CH}\mathcal{M}(k, \mathbb{Q})$ be the category of Chow motives over a field $k$ with $\mathbb{Q}$-coefficients.
Objects in $\mathrm{CH}\mathcal{M}(k)$ are given by triples $(X, p, m)$
where $X \in \mathcal{V}(k)$, $p \in \mathrm{Corr}^{0}(X, X)$ is a projector (i.e. $p \circ p = p$), and $m \in \mathbb{Z}$.
Morphisms are in $\mathrm{CH}\mathcal{M}(k)$ given by 
\[ \mathrm{Hom}_{\mathrm{CH}\mathcal{M}(k)}((X, p, m), (Y, p, n)) = q \circ \mathrm{Corr}^{n-m}(X, Y) \circ p. \]
\indent Let $M = (X, p, m), N = (Y, q, n) \in \mathrm{CH}\mathcal{M}(k)$. 
One can define a motive $M \otimes N : = (X \times Y, \pi_{X}^{*}p \cdot \pi_{Y}^{*}q, m + n)$
where $\pi_{X} : (X \times Y) \times (X \times Y) \rightarrow X \times X$ be the projection, and similar for $\pi_{Y}$.
Also, one can define $M \oplus N$.
For simplicity, we give only the definition in case $m = n$.
Then $M \oplus N : = (X \sqcup Y, p \oplus q, m)$,
 and refer to \cite[Def.~2.9 (ii),\ p.178]{Kimura} for the general case.\\ 
\indent We denote by $h(-) : \mathcal{V}(k)^{\mathrm{op}} \rightarrow \mathrm{CH}\mathcal{M}(k)$
the contravariant functor which associates to any 
$X \in \mathcal{V}(k)^{\mathrm{op}}$ its Chow motive 
\[ h(X) = (X, \Delta_{X}, 0),\]
where $\Delta_{X} \in \mathrm{Corr}^{0}(X, X)$ is the diagonal,
and to a morphism  $f : X \rightarrow Y$ the correspondence 
$h(f) = {}^{t}\Gamma_{f} \in \mathrm{Corr}^{0}(Y, X)$.
For $X, Y \in \mathcal{V}(k)$, one has $h(X \times Y) = h(X) \otimes h(Y)$.
\\
\indent Let $1 = (\mathrm{Spec}(k), \Delta_{\mathrm{Spec}(k)}, 0)$ be the unit motive and $\mathbb{L} = (\mathrm{Spec}(k), \Delta_{\mathrm{Spec}(k)}, -1)$ the Lefschetz motive.
For an non-negative integer $n$, we let 
$\mathbb{L}^{\oplus n} : = \mathbb{L} \oplus \cdot \cdot \cdot \oplus \mathbb{L}$ ($n$-times).\\
\indent Let $H^{*}$ be a fixed Weil-cohomology theory.
For $M = (X, p, m) \in \mathrm{CH}\mathcal{M}(k)$, 
one define $\mathrm{CH}^{i}(M) : = p_{*}\mathrm{CH}^{i+m}(X)$ and $H^{i}(M) : = p_{*}H^{i+2m}(X)$.\\
\indent $\mathrm{CH}\mathcal{M}(k)$ is \textit{pseudo-abelian}, i.e. 
every projector $f \in \mathrm{End}(M)$ has an image, and the canonical map
$\mathrm{Im}(f) \oplus \mathrm{Im}(\mathrm{id} - f) \rightarrow M$ is an isomorphism.
For $M = (X, p, m) \in \mathrm{CH}\mathcal{M}(k)$ and $f = p \circ f \circ p \in \mathrm{End}(M)$, one has $M = (X, p \circ f \circ p, m) \oplus (X, p - p \circ f \circ p, m)$.

\subsection{Chow-K\"unneth decompositions} \ \indent \\
Let $k$ be an algebraically closed field and $X \in \mathcal{V}(k)$ a variety of dimension $d$.
\begin{definition} (\cite{Murre On a})
We say that $X$ admits a \textit{Chow-K\"unneth decomposition} 
(CK-decomposition for short) if there exist
$\pi_{i}(X) \in \mathrm{CH}_{d}(X \times X)$ such that$:$
\begin{enumerate}
\item $\Delta_{X} = \sum_{i=0}^{2d}\pi_{i}(X)$ in $\mathrm{CH}_{d}(X \times X)$
\item $\pi_{i}(X) \circ \pi_{j}(X) = 
\begin{cases}
\pi_{i}(X) \ \ \text{if $i = j$} \\
0 \ \ \ \ \ \ \ \ \text{if $i \neq j$}
\end{cases}$
\item $cl_{X \times X}^{d}(\pi_{i}(X))$ is the $(i, 2d-i)$-th component of $\Delta_{X}$ in $H^{2d}(X \times X)$.
\end{enumerate}
\end{definition}
If $\pi_{i}$ exist, we put $h_{i}(X) : = (X, \pi_{i}(X), 0)$, and have $h(X) \cong \oplus_{i=0}^{2d}h_{i}(X)$ in $\mathrm{CH}\mathcal{M}(k, \mathbb{Q})$.
The K\"unneth components are unique in $H(X \times X)$ but $\pi_{i}(X)$ are not unique in $\mathrm{CH}_{d}(X \times X)$ in general.
For a surface, we will always use: 
\begin{prop} (\cite{Murre On the Motive}). \label{Picard motive}
Let $S \in \mathcal{V}(k)$ be a surface.
Let $P \in S$ be a closed point.
Then $S$ admits a CK-decomposition $h(S) \cong \oplus_{i = 0}^{4}h_{i}(S)$ such that$:$
\begin{enumerate}
\item $\pi_{0} : = (1/\mathrm{deg}(P))[S \times P]$ and $\pi_{4} = (1/\mathrm{deg}(P))[P \times S];$
\item There is a curve $C \subset S$ s.t. $\pi_{1}$ (resp. $\pi_{3}$) 
is supported on $S \times C$ (resp. $C \times S$);
\item $\pi_{2} : = \Delta_{S} - \pi_{0} - \pi_{1} - \pi_{3} - \pi_{4};$
\item $\pi_{i} = {}^{t}\pi_{4-i}$ for $0 \leq i \leq 4$.
\end{enumerate}
Moreover, these projectors induce isomorphisms$:$\\
\indent \ \  $(\mathrm{i}')$ $h_{0}(S) \cong 1$ and $h_{4}(S) \cong \mathbb{L} \otimes \mathbb{L}$$;$\\
\indent \ \  $(\mathrm{ii}')$ $h_{1}(S) \cong h_{1}((\mathrm{Pic}^{0}_{S/k})_{red})$ and  
$h_{3}(S) \cong h_{2 \mathrm{dim}(\mathrm{Alb}_{S/k}) - 1}(\mathrm{Alb}_{S/k}) \otimes \mathbb{L}^{2 - \mathrm{dim}(\mathrm{Alb}_{S/k})}$.
\end{prop}

\begin{prop} (\cite{Kahn and Murre and Pedrini}). \label{def of trans}
Let $S \in \mathcal{V}(k)$ be a surface.
Let $D_{i} \in \mathrm{NS}(S)_{\mathbb{Q}}$ be an orthogonal basis.
Then there is a unique splitting in $\mathrm{CH}_{2}(S \times S)$
\[ \pi_{2} = \pi_{2}^{alg} + \pi_{2}^{tr} \]
such that $\pi_{2}^{alg}  : = \sum_{i=1}^{\rho} 1/(D_{i} \cdot D_{i})[D_{i} \times D_{i}]$, 
where $(D_{i} \cdot D_{i})$ the intersection number.
Moreover, the above splitting induces a decomposition in $\mathrm{CH}\mathcal{M}(k, \mathbb{Q})$
\[ h_{2}(S) \cong h_{2}^{alg}(S) \oplus t_{2}(S) \]
such that $h_{2}^{alg}(S) : = (S, \pi_{2}^{alg}, 0) \cong \mathbb{L}^{\oplus \rho(S)}$
and $t_{2}(S)  : = (S, \pi_{2}^{tr}, 0)$.
Finally,
\[ \mathrm{CH}^{*}(h_{2}^{alg}(S)) = \mathrm{NS}(S)_{\mathbb{Q}},\ \
 \mathrm{CH}^{*}(t_{2}(S)) = T(S)_{\mathbb{Q}},\ \ H^{*}(h_{2}^{alg}(S)) = H^{2}(S)_{alg},\ \ 
 H^{2}(t_{2}(S)) = H^{2}(S)_{tr}. \]
\end{prop}
The motive $t_{2}(S)$ is called the \textit{transcendental motive} of $S$.
\subsection{Homomorphisms between transcendental motives} \ \indent \\
Here we prove some results about homomorphisms between transcendental motives.
Let $k$ be an algebraically closed field.
Let $X, Y \in \mathcal{V}(k)$ be surfaces.
We let
\begin{center}
$\mathrm{CH}_{2}(X \times Y)_{\equiv}$ $:$ 
the subgroup of $\mathrm{CH}_{2}(X \times Y)$ generated by the classes supported on subvarieties of the form 
$X \times N$ or $M \times Y$, 
with  $M$ a closed subvariety of $X$ of dimension $< 2$ 
and $N$ a closed subvariety of $Y$ of dimension $< 2$.
\end{center}
We define a homomorphism 
\begin{align*}
\Phi_{X, Y} : \mathrm{CH}_{2}(X \times Y) &\rightarrow 
\mathrm{Hom}_{\mathrm{CH}\mathcal{M}(k)}(t_{2}(X), t_{2}(Y)) \\
  \alpha &\mapsto \pi_{2}^{tr}(Y) \circ \alpha \circ \pi_{2}^{tr}(X).
\end{align*}
\begin{thm} \label{automorophism groups of motives}
(\cite[Thm.~7.4.3,\ p.165]{Kahn and Murre and Pedrini}). There is an isomorphism of groups
\[ \mathrm{CH}_{2}(X \times Y)/\mathrm{CH}_{2}(X \times Y)_{\equiv} \cong 
\mathrm{Hom}_{\mathrm{CH}\mathcal{M}(k)}(t_{2}(X), t_{2}(Y)). \]
\end{thm}
For this subsection, we need the following lemma$:$
\begin{lem} \label{composition of mod} Let $\alpha \in \mathrm{CH}_{2}(X \times Y)$ and 
$\gamma \in \mathrm{CH}_{2}(Y \times X)_{\equiv}$. Then
\[ \text{(i) $\gamma \circ \alpha  \in \mathrm{CH}_{2}(X \times X)_{\equiv}$ \ \ \ \ \ 
and \ \ \ \ \ (ii) $\alpha \circ \gamma \in \mathrm{CH}_{2}(Y \times Y)_{\equiv}$}. \]
\end{lem}
\begin{proof}
The proof of (ii) is similar to (i), so it suffices to show (i).
Without loss of generality, we may assume $\gamma$ is irreducible and supported 
on $Y \times C$ with $\mathrm{dim}(C) \leq 1$.\\
\indent First, we assume $\mathrm{dim}(C) = 0$.
Let $p \in X$ be the closed point.
For $\gamma = [Y \times p]$, then
\[\gamma \circ \alpha = [Y \times p] \circ \alpha 
= p_{YY*}^{YXY}(\alpha \times Y \cdot Y \times X \times p) 
= p_{YY*}^{YXY}(\alpha \times p) = [p_{Y*}^{YX}(\alpha) \times p]. \]
Thus $\gamma \circ \alpha \in \mathrm{CH}_{2}(X \times X)_{\equiv}$.\\
\indent Next, we assume $\mathrm{dim}(C) = 1$.
Since $\gamma$ is supported on $Y \times C$, there are a smooth irreducible curve $C$ and 
a closed embedding $\iota : C \hookrightarrow X$ such that
$\gamma = \Gamma_{\iota} \circ D \ \text{in} \ \mathrm{CH}_{2}(Y \times X)$, 
where $\Gamma_{\iota} \in \mathrm{CH}_{1}(C \times X)$ is the graph of $\iota$
and $D \in \mathrm{CH}_{2}(Y \times C)$.
Since the support of the second projection of $\Gamma_{\iota}$ has dimension $\leq 1$,
the support of the second projection of $\gamma \circ \alpha$ has dimension $\leq 1$, 
and hence $\gamma \circ \alpha \in \mathrm{CH}_{2}(X \times X)_{\equiv}$.\\
\end{proof}
Lemma \ref{composition of mod} gives the following facts. 
The first one is the functorial relation for $\Phi_{X, Y}$$:$
\begin{prop} (\cite[p.62]{Pedrini 12}). \label{functoriality of transcendental motives}
For surfaces $X$, $Y$, $Z \in \mathcal{V}(k)$, 
\[ \Psi_{Y, Z}(\beta) \circ \Psi_{X, Y}(\alpha) = \Psi_{X, Z}(\beta \circ \alpha)
\ \ \text{in} \ \ \mathrm{Hom}_{\mathrm{CH}\mathcal{M}(k)}(t_{2}(X), t_{2}(Z)). \]
\end{prop}
\begin{proof}
For example, see \cite{Kawabe}.
\end{proof}

\begin{prop} \label{bilinear}
There is a bilinear homomorphism
\begin{align*}
\circ : \frac{\mathrm{CH}_{2}(X \times Y)}{\mathrm{CH}_{2}(X \times Y)_{\equiv}} \times 
\frac{\mathrm{CH}_{2}(Y \times Z)}{\mathrm{CH}_{2}(Y \times Z)_{\equiv}} &\rightarrow 
\frac{\mathrm{CH}_{2}(X \times Z)}{\ \mathrm{CH}_{2}(X \times Z)_{\equiv}} \\
([\alpha], [\beta]) &\mapsto [\beta] \circ [\alpha] : = [\beta \circ \alpha].
\end{align*}
\end{prop}
\begin{proof}
By Lemma \ref{composition of mod}, the composition $[\beta] \circ [\alpha] : = [\beta \circ \alpha] 
\in \mathrm{CH}_{2}(X \times Z)/\mathrm{CH}_{2}(X \times Z)_{\equiv}$ is well-defined.
Thus, the assertion follows.
\end{proof}

The main proposition of this subsection is:

\begin{prop} \label{isomorphism of transcendental motives}
Let $X, Y \in \mathcal{V}(k)$ be surfaces.
Assume there are elements 
$[\alpha] \in \mathrm{CH}_{2}(X \times Y)/\mathrm{CH}_{2}(X \times Y)_{\equiv}$ and 
$[\beta] \in \mathrm{CH}_{2}(Y \times X)/\mathrm{CH}_{2}(Y \times X)_{\equiv}$ such that$:$
\begin{enumerate}
\item $[\Delta_{Y}] = [\alpha] \circ [\beta]$ in 
$\mathrm{CH}_{2}(Y \times Y)/\mathrm{CH}_{2}(Y \times Y)_{\equiv}$$;$
\item $[\Delta_{X}] = [\beta] \circ [\alpha]$ in 
$\mathrm{CH}_{2}(X \times X)/\mathrm{CH}_{2}(X \times X)_{\equiv}$.
\end{enumerate}
Then, there is an isomorphism $t_{2}(X) \cong t_{2}(Y)$ in $\mathrm{CH}\mathcal{M}(k, \mathbb{Q})$. 
\end{prop}
\begin{proof}
Assume (i).
By Proposition \ref{bilinear}, $[\Delta_{Y}] = [\alpha \circ \beta]$ in $\mathrm{CH}_{2}(Y \times Y)/\mathrm{CH}_{2}(Y \times Y)_{\equiv}$.
By Theorem \ref{automorophism groups of motives},
in $\mathrm{Hom}_{\mathrm{CH}\mathcal{M}(k)}(t_{2}(Y), t_{2}(Y))$,
\begin{equation} \label{equality for transendental 2} 
 \Phi_{Y, Y}(\Delta_{Y}) = \Phi_{Y, Y}(\alpha \circ \beta) 
 \end{equation}
Here, consider the two morphisms
\[ \Phi_{X, Y}(\alpha) \in \mathrm{Hom}_{\mathrm{CH}\mathcal{M}(k)}(t_{2}(X), t_{2}(Y)) 
\ \ \ \text{and} \ \ \ 
\Phi_{Y, X}(\beta) \in \mathrm{Hom}_{\mathrm{CH}\mathcal{M}(k)}(t_{2}(Y), t_{2}(X)).
\]
In $\mathrm{Hom}_{\mathrm{CH}\mathcal{M}(k)}(t_{2}(Y), t_{2}(Y))$, 
\begin{align*}
\Phi_{X, Y}(\alpha) \circ \Phi_{Y, X}(\beta) &= \Phi_{Y, Y}(\alpha \circ \beta)
&& \text{by Proposition \ref{functoriality of transcendental motives}} \\
&= \Phi_{Y, Y}(\Delta_{Y}) && \text{by $(\ref{equality for transendental 2})$} \\
&= \pi_{2}^{tr}(Y) && \text{by $\pi_{2}^{tr}(Y) \circ \pi_{2}^{tr}(Y) = \pi_{2}^{tr}(Y)$}\\
&= \mathrm{id}_{t_{2}(Y)}.
\end{align*}
Similarly, by (ii), we get 
$\Phi_{Y, X}(\beta) \circ \Phi_{X, Y}(\alpha) = \mathrm{id}_{t_{2}(X)}$ in 
$\mathrm{Hom}_{\mathrm{CH}\mathcal{M}(k)}(t_{2}(X), t_{2}(X))$. 
Therefore, we get $t_{2}(X) \cong t_{2}(Y)$ in $\mathrm{CH}\mathcal{M}(k, \mathbb{Q})$.
\end{proof}

\section{Torsors over group varieties}
To prove Theorem \ref{Main 3}, we prove several facts about torsors for commutative group varieties.
This section is based on \cite{Lang and Tate}, \cite[X]{Silverman}, \cite[III, §4]{Milne E}. In this section, let $K$ be a field, $\overline{K}$ an algebraic closure of $K$, 
and $K^{s}$ a separable closure of $K$.
\subsection{\textbf{Torsors}} Let $A$ be a commutative group $K$-variety and $T$ a smooth, not necessarily projective, geometrically-integral, $K$-variety.
Assume $A$ acts by the $K$-morphism 
\[ \mu : T \times A \rightarrow T. \]
\begin{definition}
We say that $T$ (or $(T, \mu)$) is a \textit{torsor} for $A/K$ if 
\[ T \times A \rightarrow T \times T \ ; \ (p, P) \mapsto (p, \mu(p, P)) \] 
is a $K$-isomorphism.
Then the action of the group $A(\overline{K})$ on $T(\overline{K})$ is simply transitive.
\end{definition}
\indent For simplicity, write $\mu(p, P)$ as $p + P$.
Here we define a \textit{subtraction map} on $T$ by 
\begin{center}
$\nu : T \times T \rightarrow A$, \\
$\nu(p, q) =$ $($the unique $P \in A(\overline{K})$ satisfying $\mu(p, P) = q$$)$.
\end{center}
Then $\nu$ is a $K$-morphism of varieties\footnote{because $\nu$ can be described as the composition of morphisms: 
$\nu = -_{E} \circ \phi \times \phi : T \times T \rightarrow E \times E \rightarrow E$}.
For simplicity, write $\nu(p, q)$ as $q-p$.\\
\indent Two torsors $(T, \mu)$, $(T', \mu')$ over $K$ for $A/K$ are isomorphic as (torsors) if there is an isomorphism $\theta : T \rightarrow T'$ of $K$-varieties that is $A$-equivariant. In other words, $\theta \circ \mu = \mu' \circ \theta \times \mathrm{id}_{A}$.
Note that any torsor is trivial after some base extension:
\begin{prop} \label{PHS}
Let $T/K$ be a torsor for $A/K$. 
Let $M/K$ be a field extension in $K^{s}$ with $T(M) \neq \emptyset$.
Choose $p_{0} \in T(M)$ and 
define the $M$-morphism 
\[ \phi^{-1} = \phi_{p_{0}}^{-1} : A_{M} \rightarrow T_{M} \ ; \ P \mapsto p_{0} + P \]
Then $\phi^{-1}$ is an isomorphism of torsors and its inverse $\phi$ is 
\[ \phi : T_{M} \rightarrow A_{M} \ ; \  p \mapsto p - p_{0} \]
In particular, a torsor $T$ is $K$-isomorphic to $A$ if and only if $T(K) \neq \emptyset$.
\end{prop}
\begin{proof}
Since $T$ is a smooth variety, $T(K^{s}) \neq \emptyset$ by \cite[Prop.~2.20, p.93]{Liu}, so such $M$ exists.
Since the action $\mu$ is defined over $K$, for all $\sigma \in \mathrm{Gal}(K^{s}/M)$, $P \in A(K^{s})$,
\[ \phi^{-1}(P)^{\sigma} = (p_{0} + P)^{\sigma} 
= p_{0}^{\sigma} + P^{\sigma} = p_{0} + P^{\sigma} = \phi^{-1}(P^{\sigma}). \]
This shows that $\phi^{-1}$ is defined over $M$.
Since $T \times A \cong T \times T$,
the induced map $\phi^{-1}(\overline{K}) : A_{M}(\overline{K}) = A(\overline{K}) \rightarrow T(\overline{K}) = T_{M}(\overline{K})$ is an isomorphism of rational points, so $\phi^{-1}$ is an isomorphism of $M$-varieties. For any $P, Q \in A_{M}(\overline{K})$, we have
\[ \phi^{-1}(P + Q) = p_{0} + (P + Q) = (p_{0} + P) + Q = \phi^{-1}(P) + Q. \]
This shows that $\phi^{-1}$ is $A$-equivariant.
Thus $\phi^{-1}$ is an isomorphism of torsors.
Hence the rest of the assertions follow.
\end{proof}

For a group $G$, let $G_{tor}$ denote the torsion subgroup.
To prove Thm.~\ref{Main 3}, we need$:$

\begin{prop} \label{existence of a torsion}
Let $T/K$ be a torsor for a commutative group variety $A/K$.
Then there are a finite Galois extension $L/K$ and a point $p_{0} \in T(L)$ such that 
\[ p_{0} - p_{0}^{\sigma} \in A(L)_{tor} \ \ \ \text{for all} \ \ \ \sigma \in \mathrm{Gal}(L/K) \]
\end{prop}
\begin{proof}
Fix a point $p \in  T(K^{s})$.
Let $n$ be an order of the element
\[ \{ a : \sigma \mapsto p - p^{\sigma} \} \in H^{1}(\mathrm{Gal}(K^{s}/K), A(K^{s})).\] 
The Kummer sequence $0 \rightarrow A(K^{s})[n] \rightarrow A(K^{s}) \overset{n}\rightarrow A(K^{s}) \rightarrow 0$
gives an exact sequence
\[ 0 \rightarrow A(K^{s})/nA(K^{s}) \rightarrow  H^{1}(\mathrm{Gal}(K^{s}/K), A(K^{s})[n]) 
\rightarrow H^{1}(\mathrm{Gal}(K^{s}/K), A(K^{s}))[n] \rightarrow 0. \]
Then, there is an element $\{ b \} \in H^{1}(\mathrm{Gal}(K^{s}/K), A(K^{s})[n])$
such that $\{ b \} = \{ a \}$ in $H^{1}(\mathrm{Gal}(K^{s}/K), A(K^{s}))$.
So there is a point $P \in A(K^{s})$ such that
\[ b(\sigma) = a(\sigma) + P  - P^{\sigma}  \in A(K^{s}) 
\ \ \ \text{for all} \ \ \ \sigma \in \mathrm{Gal}(K^{s}/K). \]
Namely, 
\[ (p - p^{\sigma}) + P - P^{\sigma} = b(\sigma) \in A(K^{s})[n] \ \ \ \text{for all}  \ \ \ 
\sigma \in \mathrm{Gal}(K^{s}/K).\]
Set $p_{0} : = p + P \in  T(K^{s})$.
Then, for all $\sigma \in \mathrm{Gal}(K^{s}/K)$,
\[ p_{0} - p_{0}^{\sigma} = (p + P) - (p + P)^{\sigma} = (p + P) - (p^{\sigma} + P^{\sigma})
= (p - p^{\sigma}) + P - P^{\sigma} = b(\sigma) \in A(K^{s})[n]. \]
Since $p_{0} \in C(K^{s})$, there is a finite Galois extension $L/K$ such that $p_{0} \in T(L)$.
Since $A(L)_{tor} = A(K^{s})_{tor} \cap A(L)$,
we get $p_{0} - p_{0}^{\sigma} \in A(L)_{tor}$ for all $\sigma \in \mathrm{Gal}(L/K)$.
\end{proof}

\subsection{Genus one curves} In this paper, we use the following terminology$:$
\begin{definition} Let $C$ be a projective, \textit{geometrically-integral}, curve over a field $K$.
\begin{enumerate} 
\item $C$ is a \textit{genus $1$ curve} if $p_{a}(C) = \mathrm{dim}\ H^{1}(X, \mathcal{O}_{C}) = 1$.
\item $C$ is an \textit{elliptic curve} if 
it is a \textit{smooth} genus $1$ curve with $C(K) \neq \emptyset$.
In other words, $C$ is an abelian $K$-variety of dimension $1$.
\end{enumerate}
\end{definition}
If $C$ is a smooth genus $1$-curve, 
then $\mathrm{Jac}(C)$ is an elliptic curve by Theorem \ref{abelian picard variety} (iii).

\begin{prop} (\cite[Prop.~6.1,\ p.54]{Schroer2010}).
Let $C$ be a genus $1$ curve over a field $K$.
Then $C$ is Gorenstein. Moreover, $\omega_{C} \cong \mathcal{O}_{C}$.
\end{prop}

For a $K$-variety $g : X \rightarrow \mathrm{Spec}(K)$, let
$X^{\#} : = \{ \ x \in X \ | \ \text{$g$ is smooth at $x$} \ \}$.

\begin{prop} \label{non-smooth genus one curve}
Assume $K = \overline{K}$. Let $C$ be a non-smooth, genus $1$ curve over $K$.
\begin{enumerate}
\item $C$ has exactly one singular point.
\item $\mathrm{char}(K) = 2$ or $3$.
\end{enumerate}
\end{prop}
\begin{proof}
(i) Let $\mu : \tilde{C} \rightarrow C$ be the normalization.
As in \cite[IV. Exercise 1.8]{Hartshorne} or \cite[p304]{Liu}, we have an exact sequence of coherent sheaves on $C$:
\[ 0 \rightarrow \mathcal{O}_{C} \rightarrow \mu_{*}\mathcal{O}_{\tilde{C}} 
\rightarrow \oplus_{c \in C}\mathcal{O}'_{C, p}/\mathcal{O}_{C, p} \rightarrow 0 \]
where $\mathcal{O}'_{C, p}$ is the integral closure of $\mathcal{O}_{C, p}$.
Then we have $p_{a}(C) = p_{a}(\tilde{C}) + \sum_{p \in C}\delta(p)$
where $\delta(p) : = \mathrm{length}(\mathcal{O}'_{C, p}/\mathcal{O}_{C, p})$.
Now, for a closed point $p \in C$,
$\delta(p) = 0$ if and only if $p$ is a smooth point.
By assumption, $p_{a}(C) = 1$, so $p_{a}(\tilde{C}) = 0$ and $\sum_{p \in C}\delta(p) = 1$.
Thus, $C$ has exactly one singular point.\\
\indent (ii) By \cite[Coro.~1,\ p.404]{Tate},
$\sum_{p \in C}\delta(p)$ is an integer multiple of $(\mathrm{char}(K) -1)/2$. 
By (i), we have $\sum_{p \in C}\delta(p) = 1$, so get $\mathrm{char}(K) = 2$ or $3$. 
\end{proof}

\begin{lem} \label{lci curve}
(Abel).
Let $C$ be a genus $1$ curve over a field $K$.
Fix $p_{0} \in C^{\#}(\overline{K})$.
Then,  there is a bijection of sets
\[ C^{\#}(\overline{K}) \rightarrow \mathrm{Pic}^{0}(C_{\overline{K}}) \ \ ; \ \ p \mapsto [p] - [p_{0}]. \]
\end{lem}
\begin{proof}
It follows from the Riemann-Roch for the Gorenstein curve $C_{\overline{K}}$.
\end{proof}

\begin{prop} \label{genus one curve}
For a genus $1$ curve $C$ over a field $K$,
$C^{\#}$ is a torsor for $\mathrm{Jac}(C)/K$.
\end{prop}
\begin{proof}
Fix $p_{0} \in C^{\#}(\overline{K})$.
By \cite[Coro.~1.8.1 (4), p.13]{MI}, there is a $K^{s}$-morphism 
\[ \phi :  C^{\#}_{K^s} \rightarrow \mathrm{Jac}(C)_{K^s}. \]
By Lemma \ref{lci curve} and \cite[Thm.~8.1,\ p.192]{Milne J}, 
the induced map $\phi(\overline{K}) : C^{\#}(\overline{K}) \rightarrow \mathrm{Jac}(C_{\overline{K}})$ is an isomorphism of rational points,
so $\phi$ is an isomorphism of group varieties.\\
\indent Define a map
\[ \mu : C^{\#} \times \mathrm{Jac}(C) \rightarrow C^{\#} \ \ ; \ \ (p, P) \mapsto \phi^{-1}(\phi(p) + P) 
= p + \phi^{-1}(P). \]
Then $\mu$ is a group action of $\mathrm{Jac}(C)$ on $C^{\#}$ over $K^s$.\\
\indent (i) First, we prove $\mu$ is defined over $K$.
For all $p \in C^{\#}(K^{s})$, $P = [q] - [p_{0}] \in \mathrm{Jac}(C)(K^{s})$, 
$\sigma \in \mathrm{Gal}(K^{s}/K)$,
\begin{align*}
\mu(p, P)^{\sigma} 
&= (p + \phi^{-1}([q]-[p_{0}]))^{\sigma} 
= p^{\sigma} + q^{\sigma}
= p^{\sigma} + \phi^{-1}([q^{\sigma}] - [p_{0}])\\
&= p^{\sigma} + \phi^{-1}([q^{\sigma}] - [p_{0}^{\sigma}])
= \mu(p^{\sigma}, P^{\sigma}). 
\end{align*}
This shows that $\mu$ is defined over $K$.\\
\indent (ii) Next, we prove 
$C^{\#} \times \mathrm{Jac}(C) \rightarrow C^{\#}  \times C^{\#} \ ; \ (p, P) \mapsto 
(p, \phi^{-1}(\phi(p) + P))$
is a $K$-isomorphism. Indeed, 
for all $p, q \in C^{\#}(\overline{K})$, one has $\mu(p, P) = q$ if and only if $\phi^{-1}(\phi(p) + P) = q$. Thus, the only choice for $P$ is $P = \phi(q) - \phi(p)$.\\
\indent By (i) and (ii), $C^{\#}$ is a torsor for $\mathrm{Jac}(C)$.
\end{proof}

The following facts will be used for the study of relations between $f$ and $j$.

\begin{prop} \label{regular compactification}
Let $C$ be a regular genus $1$ curve over a field $K$.
Let $E$ be a \textit{regular compactification} of $\mathrm{Jac}(C)$, i.e. 
it is a projective regular $K$-curve containing as a dense open subset $\mathrm{Jac}(C)$.
\begin{enumerate}
\item There is a separable field extension $M/K$ with $C_{M} \cong E_{M}$.
In particular, $C_{\overline{K}} \cong E_{\overline{K}}$.
If $C^{\#}(K) \neq \emptyset$, then $C \cong E$.
\item $E$ is also a regular genus $1$ curve over $K$.
\item There is an isomorphism $\mathrm{Jac}(C) \cong \mathrm{Jac}(E)$.
\end{enumerate}
\end{prop}
\begin{proof}
(i) By Proposition \ref{genus one curve}, $C^{\#}$ is a torsor for $\mathrm{Jac}(C)$, so
there is a finite separable field extension $M/K$ such that 
$C^{\#}_{M} \cong \mathrm{Jac}(C)_{M}$. So $C_{M} \sim_{birat} E_{M}$.
By assumption, $C$ is regular.
Since $M/K$ is separable, $C_{M}$ is also regular.
Similarly, $E_{M}$ is regular.
Since $C_{M}, E_{M}$ are regular proper curves, 
$C_{M} \cong E_{M}$.
By base change, 
$C_{\overline{K}} \cong E_{\overline{K}}$.
If $C^{\#}(K) \neq \emptyset$, then $C^{\#} \cong \mathrm{Jac}(C)$ 
by Proposition \ref{PHS}. Since $M = K$, we have $C \cong E$.
\indent (ii) By assumption, $p_{a}(C) = 1$.
By \cite[Def.~3.19,\ p.279]{Liu}, $p_{a}(C_{\overline{K}}) = 1$.
By (i), $C_{\overline{K}} \cong E_{\overline{K}}$.
By the same fact, $p_{a}(E) = p_{a}(E_{\overline{K}}) = 1$.\\
\indent (iii) If $C_{\overline{K}}$ is regular, the assertion is clear.
Assume that $C_{\overline{K}}$ is non-regular.\\
\indent (iii-i) First, we prove $\mathrm{Jac}(C) \cong E^{\#}$.
By the definition of $E^{\#}$, we get $\mathrm{Jac}(C) \subset E^{\#}$.
By (i), $p_{a}(E) = p_{a}(C) = 1$.
By Proposition \ref{non-smooth genus one curve} (i), we have
\begin{equation} \label{order of elliptic curves 1}
 | E(\overline{K}) \setminus E^{\#}(\overline{K}) | = | C(\overline{K}) \setminus C^{\#}(\overline{K}) | = 1.
\end{equation}
Here, $| - |$ denote the order of $-$.
As in (i),  $C^{\#}_{\overline{K}} \cong \mathrm{Jac}(C)_{\overline{K}}$ and 
$C_{\overline{K}} \cong E_{\overline{K}}$, so 
\begin{equation} \label{order of elliptic curves 2}
 | E(\overline{K}) \setminus \mathrm{Jac}(C)(\overline{K}) | = | C(\overline{K}) \setminus C^{\#}(\overline{K}) | = 1.
 \end{equation}
By (\ref{order of elliptic curves 1}) and (\ref{order of elliptic curves 2}), we heve  
\[ | E(\overline{K}) \setminus \mathrm{Jac}(C)(\overline{K}) | = | E(\overline{K}) \setminus E^{\#}(\overline{K}) | = 1,  \]
and hence we get $E^{\#} \cong \mathrm{Jac}(C)$.\\
\indent (iii-ii) Next, we prove $\mathrm{Jac}(E) \cong E^{\#}$.
By Proposition \ref{genus one curve}, $E^{\#}$ is a torsor for $\mathrm{Jac}(E)$.
Since $E^{\#}(K) \neq \emptyset$, $E^{\#}$ is the trivial torsor for $\mathrm{Jac}(E)$.
So $E^{\#} \cong \mathrm{Jac}(E)$.\\
\indent Combining (iii-i)  and (iii-ii) , we get $\mathrm{Jac}(C) \cong E^{\#} \cong \mathrm{Jac}(E)$.
\end{proof}
To prove Theorem \ref{Main 3}, we need the following:

\begin{prop} \label{product of elliptic}
Let $A$ be an abelian variety over a field $K$ of dimension $d$ 
and $P_{n} \in A(K)$ a torsion point of order $n \in \mathbb{Z}_{> 0}$.
\begin{enumerate}
\item Let $t : A \rightarrow A$ be the translation by $P_{n}$. Then 
$\Gamma_{t} = \Delta_{A}$ in $\mathrm{CH}_{d}(A \times A)$.
\item Let $T/K$ be a torsor for $A$ with $T(K) \neq \emptyset$.
Let $t : T \rightarrow T \ ; \ p \mapsto p + P_{n}$ be the action by $P_{n}$. Then $\Gamma_{t} = \Delta_{T}$ in $\mathrm{CH}_{d}(T \times T)$. 
\end{enumerate}
\end{prop}
\begin{proof}
(i) This is the same proof for $K = \mathbb{C}$ as in \cite[Lem.~2.1, p.573]{JY}. Let $n : A \rightarrow A$ be the multiplication by $n$.
Set $n_{A} : = (id_{A} \times n) : A \times A \rightarrow A \times A$, 
and $\mathrm{CH}^{d}_{s}(A \times A) : = 
\{ \xi \in \mathrm{CH}^{d}(A \times A) \ | \ n_{A}^{*}(\xi) = n^{2d-s}\xi \ 
\text{for all $n$} \}$.
By \cite[Coro.~2.21]{DM}, we have 
\[ \mathrm{CH}^{d}(A \times A) \cong \oplus_{s = p''}^{p'}\mathrm{CH}^{d}_{s}(A \times A). \]
This induces an isomorphism
\begin{equation} \label{n}
 n_{A}^{*} : \mathrm{CH}^{d}(A \times A) \rightarrow \mathrm{CH}^{d}(A \times A) 
\end{equation}
By $n \circ t = n$ in $\mathrm{Isom}(A)$,
we have $n_{A}^{*} \Gamma_{t} = n_{A}^{*} \Delta_{A}$,
so get $\Gamma_{t} = \Delta_{A}$ by (\ref{n}).\\
\indent (ii) Since $T(K) \neq \emptyset$, $T$ is an abelian $K$-variety.
Let $n : T \rightarrow T$ be the multiplication by $n$.
Regarding $t$ as an isomorphism of trivial torsors for $T$, we have $n \circ t = n$ in $\mathrm{Isom}(T)$. By the same argument of (i), $\Gamma_{t} = \Delta_{T}$.
\end{proof}

\section{Chow motives of torsors under abelian varieties}
The purpose of this section is to prove the following:
\begin{thm} \label{motive of smooth genus one}
Let $K$ be an arbitrary field.
Let $A$ be an abelian variety over $K$ of dimension $d$.
Let $T$ be a torsor over $K$ under $A$.\footnote{First, the author only proved Theorem \ref{motive of smooth genus one} for the case of smooth genus $1$ curves under the guidance of Prof. Hanamura. Thanks to the referee's suggestion, he proved the above theorem.}
Then there is an isomorphism 
\[ h(T) \cong h(A) \] 
in the category $\mathrm{CH}\mathcal{M}(K, \mathbb{Q})$ of Chow motives.
\end{thm}
\textbf{Proof of Theorem \ref{motive of smooth genus one}.} \ \indent \\
By definition, we have to prove$:$
there are elements $a \in \mathrm{CH}_{d}(T \times A)$, 
$b \in \mathrm{CH}_{d}(A \times T)$ such that
\[ 
\begin{cases}
a \circ b = \Delta_{A}  &  \text{in \ \ $\mathrm{CH}_{d}(A \times A)$} \\
b \circ a = \Delta_{T} &  \text{in \ \ $\mathrm{CH}_{d}(T \times T)$}
\end{cases}
\]

\textbf{Step 1. Construct correspondences on the varieties $A$ and $T$.}
By Proposition \ref{existence of a torsion},
there are a finite Galois extension $L/K$ and a point $p_{0} \in T(L)$ such that 
\[ p_{0} - p_{0}^{\sigma} \in A(L)_{tor} \ \ \ \text{for all} \ \ \ \sigma \in \mathrm{Gal}(L/K). \]
Let $n$ be the degree of $L/K$ and let $G : = \mathrm{Gal}(L/K)$.
By definition of a torsor,
there is an isomorphism of $L$-varieties
\[ \phi = \phi_{p_{0}}  : T_{L} \rightarrow A_{L}. \]
Let $p : T_{L} \rightarrow T$ and $q : A_{L} \rightarrow A$ be the projections.
Let 
\[ \Gamma_{\phi} \in \mathrm{CH}_{1}(T_{L} \times_{L} A_{L}), \ \ \ 
\Gamma_{p} \in \mathrm{CH}_{1}(T_{L} \times_{K} T), \ \ \text{and} \ \ 
\Gamma_{q} \in \mathrm{CH}_{1}(A_{L} \times_{K} A). \]
be the graph of $\phi$, $p$, and $q$, respectively.
We define
\begin{align*}
a &: =  (1/n) \ \Gamma_{q} \circ \Gamma_{\phi} \circ {}^{t}\Gamma_{p} 
\in \mathrm{CH}_{1}(T \times A) \\
b &: = (1/n) \ \Gamma_{p} \circ {}^{t}\Gamma_{\phi} \circ {}^{t}\Gamma_{q} 
\in \mathrm{CH}_{1}(A \times T)
\end{align*} 
\indent \textbf{Step 2. Translations.} To prove $a \circ b = \Delta_{A}$ and $b \circ a = \Delta_{T}$, we prove an elementary lemma. For any $\sigma \in G$, let $\sigma : \mathrm{Spec}(L) \rightarrow \mathrm{Spec}(L)$ denote the induced morphism.\\
\indent For any $\sigma \in G$, we define
\begin{align*}
\phi^{\sigma} \circ \phi^{-1} &: = (\mathrm{id}_{A} \times_{K} \sigma) \circ \phi \circ 
(\mathrm{id}_{T} \times_{K} \sigma^{-1}) \circ \phi^{-1} \in \mathrm{Isom}(A_{L}/L) \\
(\phi^{-1})^{\sigma} \circ \phi &: = (\mathrm{id}_{T} \times_{K} \sigma) \circ \phi^{-1} \circ 
(\mathrm{id}_{A} \times_{K} \sigma^{-1}) \circ \phi \in \mathrm{Isom}(T_{L}/L) 
\end{align*}
For short, we denote by $p_{0} - p_{0}^{\sigma} : A_{L} \rightarrow A_{L}$ the translation by 
$p_{0} - p_{0}^{\sigma} \in A(L)_{tor}$ and denote by $p_{0}^{\sigma}  - p_{0} : T_{L} \rightarrow T_{L}$ the action by $p_{0}^{\sigma} - p_{0} \in A(L)_{tor}$.

\begin{lem} \label{aut and translation} For any $\sigma \in G$, we have
\[ (\mathrm{i}) \ \phi^{\sigma} \circ \phi^{-1}  = p_{0} - p_{0}^{\sigma} \ \ \text{in} \ \ \mathrm{Isom}(A_{L}/L). \ \ \ \ \ \ (\mathrm{ii}) \ (\phi^{-1})^{\sigma} \circ \phi = p_{0}^{\sigma} - p_{0} \ \ \text{in} \ \ \mathrm{Isom}(T_{L}/L). \]
\end{lem}
\begin{proof}
(i) It suffices to prove $(\phi^{\sigma} \circ \phi^{-1})(K^{s})  = (p_{0} - p_{0}^{\sigma})(K^{s})$ 
in $\mathrm{Isom}(A(K^{s}))$.\\
The morphisms $\phi$ and $\phi^{-1}$ induce isomorphisms of $K^{s}$-rational points
\begin{align*} 
\phi(K^{s}) : T(K^{s}) \rightarrow  A(K^{s}) 
\ \ &; \ \ p \mapsto p - p_{0} \\
\phi^{-1}(K^{s}) : A(K^{s}) \rightarrow T(K^{s}) 
\ \ &; \ \ P \mapsto p_{0} + P.
\end{align*}
Thus, $\phi^{\sigma} \circ \phi^{-1}$ induces an isomorphism of $K^{s}$-rational points
\[
A(K^{s}) \ni P  \overset{\phi^{-1}(K^{s})}\mapsto p_{0} + P 
\overset{\sigma^{-1}}\mapsto p_{0}^{\sigma^{-1}} + P^{\sigma^{-1}}  
\overset{\phi(K^{s})}\mapsto (p_{0}^{\sigma^{-1}} + P^{\sigma^{-1}}) - p_{0} \overset{\sigma}\mapsto
(p_{0} + P) - p_{0}^{\sigma} 
\in A(K^{s})
\]
Let $O \in A$ be the origin. We have in $A(K^{s})$
\[
(p_{0} + P) - p_{0}^{\sigma} 
= (p_{0} + P) - (p_{0}^{\sigma} + O) 
= (p_{0} - p_{0}^{\sigma}) + P - O = (p_{0} - p_{0}^{\sigma}) + P. \]
Therefore, we get 
$\phi^{\sigma} \circ \phi^{-1} = p_{0}-p_{0}^{\sigma}$ in $\mathrm{Isom}(A_{L}/L)$.\\
\indent (ii) Similarly, $(\phi^{-1})^{\sigma} \circ \phi$ induces
an isomorphism of $K^{s}$-rational points
\[
T(K^{s}) \ni p  \overset{\phi(K^{s})}\mapsto p - p_{0} 
\overset{\sigma^{-1}}\mapsto p^{\sigma^{-1}} - p_{0}^{\sigma^{-1}}  
\overset{\phi^{-1}(K^{s})}\mapsto p_{0} + (p^{\sigma^{-1}} - p_{0}^{\sigma^{-1}})  \overset{\sigma}\mapsto
p_{0}^{\sigma} + (p - p_{0}^{\sigma}) 
\in T(K^{s})
\]
As the action $T \times A \rightarrow T$ is equal to $\phi^{-1} \circ +_{A}  \circ \phi \times \mathrm{id_{A}}$, we have in $T(K^{s})$
\[ p_{0}^{\sigma} + (p - p_{0}) = p_{0} + \phi(p)
= \phi^{-1}(\phi(p_{0}^{\sigma}) + \phi(p))
= \phi^{-1}(\phi(p) + \phi(p_{0}^{\sigma})) = p + \phi(p_{0}^{\sigma})
= p + (p_{0}^{\sigma} - p_{0}). \] 
Therefore, we get $(\phi^{-1})^{\sigma} \circ \phi = p_{0}^{\sigma} - p_{0}$ in $\mathrm{Isom}(T_{L}/L)$.
\end{proof}

\textbf{Step~3. Calculate the correspondences on the varieties $A$ and $T$.}\\
First, we prove $a \circ b = \Delta_{A}$. Let
\[ 
\xymatrix{
& & \\
A_{L} \times_{L} A_{L} \ar@/^18pt/[rr]^{q' \times q'} \ar[r] & 
A_{L} \times_{K} A_{L} \ar[r]_{q \times q} & A \times A.
}
\]
In $\mathrm{CH}_{d}(A \times A)$,  
\begin{align*}
a \circ b &= (1/n^{2}) \ (\Gamma_{q} \circ \Gamma_{\phi} \circ {}^{t}\Gamma_{p}) 
\circ (\Gamma_{p} \circ {}^{t}\Gamma_{\phi} \circ {}^{t}\Gamma_{q}) 
&& \text{by $a = \Gamma_{q} \circ \Gamma_{\phi} \circ {}^{t}\Gamma_{p}$ and 
$b = \Gamma_{p} \circ {}^{t}\Gamma_{\phi} \circ {}^{t}\Gamma_{q}$} \\
&= (1/n^{2}) \ \Gamma_{q} \circ \Gamma_{\phi} \circ \sum_{\sigma^{-1} \in G}\Gamma_{\sigma^{-1}} 
\circ {}^{t}\Gamma_{\phi} \circ {}^{t}\Gamma_{q}  
&& \text{by Proposition \ref{Galois extension correspondence} (ii)} \\
&= (1/n^{2}) \sum_{\sigma^{-1} \in G}\ \Gamma_{q} \circ \Gamma_{\sigma} \circ \Gamma_{\phi} \circ \Gamma_{\sigma^{-1}} \circ {}^{t}\Gamma_{\phi} \circ {}^{t}\Gamma_{q}
&&\text{by $q \circ \sigma = q$} \\
&= (1/n^{2}) \sum_{\sigma \in G} \ \Gamma_{q} \circ 
\Gamma_{\sigma \circ \phi \circ \sigma^{-1} \circ \phi^{-1}} \circ {}^{t}\Gamma_{q}  \\
&= (1/n^{2}) \sum_{\sigma \in G} \ \Gamma_{q} \circ \Gamma_{p_{0}-p_{0}^{\sigma}} \circ {}^{t}\Gamma_{q} 
&& \text{by Lemma \ref{aut and translation} (i)} \\
&= (1/n^{2}) \sum_{\sigma \in G} \ (q \times q)_{*}(\Gamma_{p_{0}-p_{0}^{\sigma}})
&& \text{by Lieberman's lemma} \\
&= (1/n^{2}) \sum_{\sigma \in G} \ (q' \times q')_{*}(\Gamma_{p_{0}-p_{0}^{\sigma}}) \\
&= (1/n^{2}) \ (q' \times q')_{*}(n\Delta_{A_{L}}) 
&& \text{by Proposition \ref{product of elliptic} (i)} \\
&= \Delta_{A} && \text{by the proper push-forward of the cycle}
\end{align*}
\indent Next, we prove $b \circ a = \Delta_{T}$.
Indeed, $T_{L}$ is an abelian $L$-variety since $T(L) \neq \emptyset$.
Replacing $\phi^{\sigma} \circ \phi^{-1}$ with $(\phi^{-1})^{\sigma} \circ \phi$ (Lemma \ref{aut and translation} (i) with (ii)) in the previous discussion, we get $b \circ a = \Delta_{T}$ in $\mathrm{CH}_{d}(T \times T)$.
Hence, we get $h(T) \cong h(A)$ in $\mathrm{CH}\mathcal{M}(K, \mathbb{Q})$.
\begin{rem}
The statement of Theorem \ref{motive of smooth genus one} is mentioned in \cite[Rem.~1.6, p.62]{FV}. 
\end{rem}

\section{Genus one fibrations and Jacobian fibrations}
In this section, let $k$ be an algebraically closed field of arbitrary characteristic.\\
The purpose of this section is to prove the following$:$
\begin{thm} \label{aim of jacobian}
Let $f : X \rightarrow C$ be a minimal genus $1$ fibration over $k$ and 
$j : J \rightarrow C$ the Jacobian fibration of $f$. Then there are isomorphisms of Chow motives
\begin{enumerate}
\item $h_{1}(X) \cong h_{1}(J)$, $h_{3}(X) \cong h_{3}(J)$.
\item  $h_{2}^{alg}(X) \cong h_{2}^{alg}(J)$.
\end{enumerate}
\end{thm}
The results of this section are based on 
\cite{Cossec and Dolgachev} and \cite{Cossec and Dolgachev and Liedtke and Kondo}.

\subsection{Fibrations}
First of all, we fix the terminologies, so begin with:
\begin{definition}
Let $X, Y \in \mathcal{V}(k)$.
A morphism $f : X \rightarrow Y$ is a \textit{fibration} if it is a proper surjective morphism  
such that $f_{*}\mathcal{O}_{X} \cong \mathcal{O}_{Y}$.
\end{definition}

\begin{prop} \label{pullback injective} 
For a fibration $f : X \rightarrow Y$, $f^{*} : \mathrm{Pic}(Y) \rightarrow \mathrm{Pic}(X)$ is injective.
\end{prop}
\begin{proof}
Let $\mathcal{L} \in \mathrm{Pic}(Y)$ with $f^{*}\mathcal{L} \cong \mathcal{O}_{X}$.
By projection formula,
\[ \mathcal{L} \cong  \mathcal{L}  \otimes \mathcal{O}_{Y} \cong  \mathcal{L} \otimes f_{*}\mathcal{O}_{X} 
\cong  f_{*}(f^{*}\mathcal{L} \otimes \mathcal{O}_{X})  \cong  f_{*}(f^{*}\mathcal{L}) \cong  f_{*}\mathcal{O}_{X}  \cong  \mathcal{O}_{Y}. \]
\end{proof}

\indent Let $f : X \rightarrow C$ be a fibration from a surface to a curve.
Then $f$ is flat, and all fibers of $f$ are connected
(e.g. \cite[III. Prop.~9.7\ and\ Coro.~11.3]{Hartshorne}).
We consider a fiber $X_{c}$ of $f$ over a closed point $c \in C$
as an effective Cartier divisor with the sheaf of ideals $\mathcal{O}_{X}(-X_{c}) = f^{*}(\mathcal{O}_{C}(-c))$.  
Since $X$ is regular, we can identify $X_{c}$ with corresponding Weil divisor, and write the fiber
\[ X_{c} = \sum_{i = 1}^{r}m_{i} E_{i} \]
as the finite sum of its irreducible components.
The number $m_{i}$ is called the \textit{multiplicity} of the component $E_{i}$,
the number $m_{c} : = \mathrm{gcd}(m_{1}, \cdot \cdot \cdot , m_{r})$
is called \textit{multiplicity} of $X_{c}$, and 
the fiber $X_{c}$ is called \textit{multiple} (resp. \textit{non-multiple}) if $m_{c} > 1$ (resp. $m_{c} = 1$).\\
For any fiber $X_{c}$, we denote by $\overline{X_{c}}$ the divisor $1/m_{c} \cdot X_{c}$.
Then $X_{c} = m_{c} \overline{X_{c}}$.\\
\indent Let $S$ be a surface.
Let $D$ and $D'$ be divisors on $S$.
We denote by $(D \cdot D')$ the intersection number of $D$ and $D'$.
The divisor $D$ is \textit{numerically equivelent} to $D'$ $($for short $D \equiv D'$$)$ 
if $(D \cdot C) = (D' \cdot C)$ for any curve $C$ on $S$.

\subsection{Genus one fibrations}
Let $k$ be an algebraically closed field.
Let 
\[ f : X \rightarrow C\]
be a minimal genus 1 fibration from a surface over $k$.
Since $f$ is a fibration, it is proper and flat.
By general properties of morphism of schemes, 
all geometric fibers are geometrically-connected and there is a dense open subset $U$ of $C$
such that an \textit{elliptic} (resp. a \textit{quasi-elliptic}) fibration $f$ is 
\textit{smooth} (resp. \textit{geometrically-integral}) over $U$.
Let $\Sigma$ be the finite set of closed points $c \in C$ such that 
the scheme-theoretical fiber $X_{c}$ is \textit{not-smooth} (resp. \textit{not-integral}) 
if $f$ is \textit{elliptic} (resp. $f$ is \textit{quasi-elliptic}).
We call $X_{c}$, $c \in \Sigma$, the \textit{singular fibers} of $f$.\\
\indent The following formula is well-known and is very useful$:$

\begin{thm} \label{Canonical bundle formula}
$($canonical bundle formula$)$.
Let $f : X \rightarrow C$ be a genus $1$ fibration.\\
Let $R^{1}f_{*}\mathcal{O}_{X} = \mathcal{L} \oplus T$ be the decomposition,
with $L$ an invertible sheaf on $C$ and $T$ an $\mathcal{O}_{C}$-module of finite length.
Then
\[ \omega_{X} \cong f^{*}(\mathcal{L}^{-1} \otimes \omega_{C}) \otimes 
\mathcal{O}_{X}(\sum_{i = 1}^{r}n_{i} \overline{X_{c_{i}}}),\]
where
\begin{enumerate}
\item $m_{i} \overline{X_{c_{i}}} =X_{c_{i}}$ $(c_{i} \in C)$ are all the multiple fibers of $f$,
\item  $0 \leq n_{i} < m_{i}$,
\item  $n_{i} = m_{i} - 1$ if $c_{i}$ is not supported in $T$, and
\item  $\mathrm{deg}(\mathcal{L}^{-1} \otimes \omega_{C}) = 2  p_{a}(C) - 2 + \chi(\mathcal{O}_{X}) + \mathrm{length}(T)$.
\end{enumerate}
\end{thm}
\begin{proof}
For example, see \cite[Thm.~2,\ p.27]{Bombieri and Mumford II} or \cite[Thm.~7.15,\ p.100]{Badescu}.
\end{proof}

In the rest of this subsection, we refine the relative Mordell-Weil theorem
for the proof of Proposition \ref{Main 2}.
Let $f : X \rightarrow C$ be a genus $1$ fibration.
For any point $c \in C$ (not necessarily closed), we denote by 
\[ r_{c} : \mathrm{Pic}(X) \rightarrow \mathrm{Pic}(X_{c}) \] 
the homomorphism obtained by restriction of invertible sheaves on $X$ to $X_{c}$.
We set
\begin{align*}
\mathrm{Pic}(X)_{0} : &= \mathrm{Ker}(\mathrm{deg} \circ r_{\eta})\\
                             &= \{ \ \mathcal{L} \in \mathrm{Pic}(X) \ | \ \mathrm{deg}\circ r_{c}(\mathcal{L}) = 0 \ \ \text{for  any} \ c \in C \ \} \subset \mathrm{Pic}(X)
\end{align*}
Here, the second equality follows from the function $c \mapsto \chi(X_{c}, r_{c}(\mathcal{L}))$ is constant on $C$ (\cite[Coro.~(b), p.50]{Mumford A}).
Then $\mathrm{Pic}(X)_{0} \supset  (\mathrm{Pic}^{0}_{X/k})_{red}(k)$.
We define 
\begin{align*}
\mathrm{Pic}(X)_{f} &: = \mathrm{Ker}(r_{\eta}) \subset \mathrm{Pic}(X)_{0}  \\ 
\mathcal{E}(C) &: = \mathrm{Pic}(X)_{f}/f^{*}\mathrm{Pic}(C)  
\subset \mathrm{Pic}(X)/f^{*}\mathrm{Pic}(C)
\end{align*} 
For a singular fiber $X_{c} = \sum_{i=1}^{r}n_{i} E_{i} \in \mathrm{Div}(X)$ with $E_{i} \neq E_{j}$ for $i \neq j$,  set 
$\mathrm{Num}_{c}(X) : = \sum_{i=1}^{r}\mathbb{Z}[E_{i}]$ in $\mathrm{Num}(X)$.
To refine the MW theorem, we need:
\begin{prop}
(\cite[Prop.~5.2.1,\ p.293]{Cossec and Dolgachev}). 
\label{prepare Mordell-Weil Theorem for function fields} 
Let $f : X \rightarrow C$ be a genus $1$ fibration.
\begin{enumerate}
\item $\mathcal{E}(C) = \oplus_{c \in \Sigma}((\mathrm{Num}_{c}(X)/\mathbb{Z}[\overline{X_{c}}]) \oplus \mathbb{Z}/m_{c}\mathbb{Z})$, where $\Sigma$ is the set of singular points of $f$.
\item $\mathrm{Pic}(X)_{f}/f^{*}\mathrm{Pic}^{0}(C)$ is a finitely-generated abelian group.
\item $\mathrm{Pic}(X)_{0}/\mathrm{Pic}(X)_{f} \cong \mathrm{Jac}(X_{\eta})(K)$.
\end{enumerate}
\end{prop}
\begin{proof}
(i) By the local exact sequence $\oplus_{c}\mathrm{CH}_{1}(X_{c}) \overset{i_{*}}\rightarrow \mathrm{CH}_{1}(X) \overset{i^{*}_{\eta}}\rightarrow \mathrm{CH}_{0}(X_{\eta}) \rightarrow 0$,    
we have $\mathrm{Pic}(X)_{f} = \sum_{c}(\sum_{i} \mathbb{Z}[E_{i}])$, 
where $c$ runs over all closed points of $C$ and $E_{i}$ are irreducible components of $X_{c}$.
Note 
\[ [D] \in f^{*}\mathrm{Pic}(C) \ \ \text{if and only if} \ \ [D] = \sum n_{c} [X_{c}] \ \ ( n_{c} \in \mathbb{Z} ).\]
For a closed point $c \in C$, let $\mathcal{E}_{c}(C)$ be the subgroup of $\mathcal{E}(C)$ generated by the images of irreducible components of $X_{c}$.
In other words, if $\mathrm{Pic}_{c}(X)$ is the subgroup of $\mathrm{Pic}(X)$ generated by the images of  irreducible components of $X_{c}$, then
\[ \mathcal{E}_{c}(C) = \mathrm{Pic}_{c}(X)/(\mathrm{Pic}_{c}(X) \cap f^{*}\mathrm{Pic}(C)) \cong 
(\mathrm{Pic}_{c}(X) + f^{*}\mathrm{Pic}(C))/f^{*}\mathrm{Pic}(C). \]
\indent (i-i) We have $\mathcal{E}(C) = \oplus_{c}\mathcal{E}_{c}(C)$.
Indeed, assume $D_{c} \in \mathrm{Div}(X)$ are divisors supported on $X_{c}$,
and $\sum D_{c} = 0$ in $\mathcal{E}(C)$.
Then  $\sum_{c}D_{c} \sim f^{*}(\delta)$ for some $\delta \in \mathrm{Div}(C)$, namely
$\sum_{c}D_{c} - f^{*}(\delta) = \mathrm{div}(\phi)$
for some $\phi \in k(X)^{\times}$.
Since $\mathrm{div}(\phi)$ has support on fibers,
there is a non-empty open set $U$ of $C$ such that $\phi$ is regular on $f^{-1}U;$
noting $f_{*}\mathcal{O}_{X} \cong \mathcal{O}_{C}$ one has
$\phi \in \mathcal{O}_{X}(f^{-1}U) = \mathcal{O}_{C}(U)$
so $\phi = f^{*}\psi$ for some $\psi \in k(C)^{\times}$.
Therefore $\sum_{c}D_{c}  = f^{*}(\delta + \mathrm{div}(\psi))$;
rewriting $\delta$ for $\delta + \mathrm{div}(\psi)$, 
we have $\sum_{c}D_{c} = f^{*}(\delta)$ in $\mathrm{Div}(X)$.
If $\delta = \sum n_{c}[c]$, one has $D_{c} = n_{c}[X_{c}]$ for each $c$, 
in particular $D_{c} = 0$ in $\mathcal{E}(C)$.\\
\indent (i-ii) For a closed point $c \in C$, 
let $E_{1}, \cdot \cdot \cdot , E_{r}$ be the irreducible components of $X_{c}$.
Then 
$\mathcal{E}_{c}(C) \cong \mathbb{Z}\{ E_{i} \}/\mathbb{Z}[X_{c}]$,
where $\mathbb{Z}\{ E_{i} \}$ is the free abelian group with basis $\{ E_{i} \}$.
Indeed, there is a surjection 
$\mathbb{Z}\{ E_{i} \} \rightarrow \mathcal{E}_{c}(C)$.
If $D \in \mathbb{Z}\{ E_{i} \}$ maps to zero in $\mathcal{E}(C)$, 
the argument in (i) shows $D = n_{c}[X_{c}]$ for $n_{c} \in \mathbb{Z}$.\\
\indent (i-iii) Recall that $X_{c}$ is a primitive generator of $\mathbb{Z} \cdot X_{c}$, 
so $X_{c} = m_{c} \overline{X_{c}}$.
One has obviously a short exact sequence 
\[ 0 \rightarrow \mathbb{Z}/m_{c}\mathbb{Z} \rightarrow \mathbb{Z}\{ E_{i} \}/\mathbb{Z}[X_{c}] \rightarrow \mathbb{Z}\{ E_{i} \}/\mathbb{Z}[\overline{X_{c}}] \rightarrow 0.\]
The last group is free of rank $r - 1$, so there is a non-canonical splitting of this short exact sequence.
Recall $\mathrm{Num}_{c}(X)$ is the subgroup of $\mathrm{Num}(X)$ 
generated by the classes of irreducible components of $X_{c}$.\\
\indent (i-iv) There is an isomorphism
$\mathbb{Z}\{ E_{i} \}/\mathbb{Z}[\overline{X_{c}}] \rightarrow \mathrm{Num}_{c}(X)/\mathbb{Z}[\overline{X_{c}}]$.
Indeed, there is a surjection $\mathbb{Z}\{ E_{i} \} \rightarrow \mathrm{Num}_{c}(X)$, which induces a surjection  
$\mathbb{Z}\{ E_{i} \} \rightarrow \mathrm{Num}_{c}(X)/\mathbb{Z}[\overline{X_{c}}]$.
Assume $D = \sum n_{i}E_{i}$ goes to 0 by this homomorphism$;$ then one has 
$D - n \overline{X_{c}} \equiv 0$, so $(D^{2}) = 0$ by Chow's moving lemma, 
so $D \in \mathbb{Z}\overline{X_{c}}$ by \cite[Coro.~2.6,\ p.19]{Badescu}.\\
\indent Combining (i-i)-(i-iv), we have $\mathcal{E}(C) = \oplus_{c}\mathcal{E}_{c}(C)$ with non-canonical isomorphisms
\[ \mathcal{E}_{c}(C) \cong (\mathrm{Num}_{c}(X)/\mathbb{Z}[\overline{X_{c}}]) 
\oplus \mathbb{Z}/m_{c}\mathbb{Z}. \]
In particular, if $X_{c}$ is integral, then $\mathcal{E}_{c}(C) = 0$.\\
\indent (ii) For two points $c, c'$ of $C$, one has $[c] - [c'] \in \mathrm{Pic}^{0}(C)$, 
thus $[X_{c}] - [X_{c'}] \in f^{*}\mathrm{Pic}^{0}(C)$.\\
\indent First, assume $f$ has a singular fiber.
If we choose a singular point $c'$,
in the group $\mathrm{Pic}(X)_{f}/f^{*}\mathrm{Pic}^{0}(C)$,
any smooth fiber $[X_{c}]$ equals the singular fiber $[X_{c'}]$.
Since $\mathrm{Pic}(X)_{f}$ is generated by the irreducible components of all fibers,
it follows that $\mathrm{Pic}(X)_{f}/f^{*}\mathrm{Pic}^{0}(C)$ is generated by the irreducible components 
of the singular fibers only. Next, assume $f$ has no singular fibers.
If we choose a smooth point $c'$, 
in the group $\mathrm{Pic}(X)_{f}/f^{*}\mathrm{Pic}^{0}(C)$,
any smooth fiber $[X_{c}]$ equals the smooth fiber $[X_{c'}]$.
Thus $\mathrm{Pic}(X)_{f}/f^{*}\mathrm{Pic}^{0}(C)$ is generated by the smooth fiber only.\\
\indent (iii) 
By Tsen's thm.(\cite[Thm.~6.2.8,\ p.143]{Gille and Szamuely}), $\mathrm{Br}(K) = 0$.
By Proposition \ref{exact brauer},
\begin{equation} \label{Pic = Jac}
\mathrm{Pic}^{0}(X_{\eta}) \cong \mathrm{Jac}(X_{\eta})(K)  
\end{equation}
By definition, $\mathrm{Pic}(X)_{0} = \mathrm{Ker}(\mathrm{Pic}(X) \overset{r_{\eta}}\rightarrow \mathrm{Pic}(X_{\eta}) \overset{\mathrm{deg}}\rightarrow \mathbb{Z})$, 
and hence the restriction 
$r_{\eta}^{0} : \mathrm{Pic}(X)_{0} \rightarrow \mathrm{Pic}^{0}(X_{\eta})$
is well-defined.
By the restriction $r_{\eta} : \mathrm{Pic}(X) \rightarrow \mathrm{Pic}(X_{\eta})$ is surjective, so is $r_{\eta}^{0}$.
Thus, there is an isomorphism of groups 
\begin{equation} \label{homomorphism between Pic}
\mathrm{Pic}(X)_{0}/\mathrm{Pic}(X)_{f} \cong \mathrm{Pic}^{0}(X_{\eta})   
\end{equation}
Combining $(\ref{Pic = Jac})$ and $(\ref{homomorphism between Pic})$, we get an isomorphism $\mathrm{Pic}(X)_{0}/\mathrm{Pic}(X)_{f} \cong \mathrm{Jac}(X_{\eta})(K)$.
\end{proof}
To state a refined MW theorem, we recall a notion:
Let $K/k$ be a primary field extension and $A$ an abelian $K$-variety.
We say that a couple $(\mathrm{Tr}_{K/k}(A), \tau)$ consisting of an abelian $k$-variety $\mathrm{Tr}_{K/k}(A)$ and a $K$-homomorphism $\tau : \mathrm{Tr}_{K/k}(A)_{K}
\rightarrow A$ is a \textit{$K/k$-trace of $A$} if $\tau$ has finite kernel and satisfies the following condition. For an abelian $k$-variety $B$ and a $K$-homomorphism $\beta : B_{K} \rightarrow A$, there is a $k$-homomorphism 
$\beta' : B \rightarrow \mathrm{Tr}_{K/k}(A)$ such that 
$\beta = \tau \circ \beta'_{K}$. Indeed, it exists (\cite[Thm.~8, p.213]{Lang}).\\
\indent The following theorem is a refined version of the MW (or the Lang-N\'eron) theorem.
We focus on the $K/k$-trace of $\mathrm{Jac}(X_{\eta})$ that is isogenous to a sub-abelian variety of $(\mathrm{Pic}_{X/k}^{0})_{red}$ for the proof of Proposition \ref{Main 2}:

\begin{thm}
\label{Mordell-Weil Theorem for function fields}
Let $f : X \rightarrow C$ be a genus $1$ fibration.
Let $T$ be an abelian $k$-variety of dimension $\leq 1$ that is 
the $K/k$-trace $(\mathrm{Tr}_{K/k}(\mathrm{Jac}(X_{\eta})), \tau)$ 
of $\mathrm{Jac}(X_{\eta})$ if $f$ is elliptic and is $0$ if $f$ is quasi-elliptic.
Then 
\begin{enumerate}
\item There is an isogeny of abelian $k$-varieties 
$(\mathrm{Pic}_{X/k}^{0})_{red} \sim_{isog} T \times \mathrm{Pic}_{C/k}^{0}$. 
\item $T = 0$ $\Longleftrightarrow$ $b_{1}(X) = b_{1}(C)$ \\
\indent \ \ \ \ \ \ $\Longleftrightarrow$ $\mathrm{Jac}(X_{\eta})(K)/\tau(T)(k) = \mathrm{Jac}(X_{\eta})(K)$ is f.g. \ \ \ (Lang-N\'eron) \\
Moreover, if these conditions are satisfied, then
\[
 \mathrm{rank}(\mathrm{Jac}(X_{\eta})(K)) = \rho(X) - 2 - \sum_{c \in C}(\mathrm{\# \mathrm{Irr}}(X_{c}) - 1) \ \ \ \text{$($Shioda-Tate formula$)$} \]
Here, $\mathrm{Irr}(X_{c})$ is the set of the distinct irreducible components of the singular fiber $X_{c}$.
\item $T \neq 0$ $\Longleftrightarrow$ $b_{1}(X) = b_{1}(C) + 2$\\
\indent \ \ \ \ \ \ $\Longleftrightarrow$ $\mathrm{Jac}(X_{\eta})(K)$ is an not finitely-generated abelian group\\
\indent \ \ \ \ \ \ $\Longleftrightarrow$ there is an isogeny of elliptic $K$-curves 
\[ \psi : T \times_{\mathrm{Spec}(k)} \mathrm{Spec}(K) \rightarrow \mathrm{Jac}(X_{\eta}). \]
\end{enumerate}
In particular, if $f$ is quasi-elliptic, then $\mathrm{Jac}(X_{\eta})(K)$ is finitely-generated.
\end{thm}
\begin{proof}
This proof is based on \cite[Prop.~5.2.1 (v),\ p.293]{Cossec and Dolgachev}.
By Proposition \ref{pullback injective},
there is an injective morphism of group schemes 
$f^{*} : \mathrm{Pic}_{C/k} \hookrightarrow \mathrm{Pic}_{X/k}$.
So we get an injection 
\[ f^{*} : \mathrm{Pic}_{C/k}^{0} = (\mathrm{Pic}_{C/k}^{0})_{red} \hookrightarrow (\mathrm{Pic}_{X/k}^{0})_{red}. \]
By Theorem \ref{abelian picard variety}, $\mathrm{Pic}_{C/k}^{0}$ and  
$(\mathrm{Pic}_{X/k}^{0})_{red}$ are abelian $k$-varieties.
By Poincare's reducibility theorem, 
there is an abelian $k$-subvariety $A \subset (\mathrm{Pic}_{X/k}^{0})_{red}$ such that 
\begin{equation} \label{isogeny for A}
(\mathrm{Pic}_{X/k}^{0})_{red} \sim_{isog} A \times \mathrm{Pic}_{C/k}^{0}.
\end{equation}
\indent First, we prove that $A$ satisfies assertions (i)-(iii).
By (\ref{isogeny for A}), it remains to prove $\mathrm{dim}(A) \leq 1$ and 
$A$ satisfies assertions (ii) and (iii).\\
\indent Assume $A = 0$.
Then $\mathrm{Pic}_{C/k}^{0} \sim_{isog} (\mathrm{Pic}_{X/k}^{0})_{red}$. 
So $\mathrm{Pic}_{C/k}^{0}(k) = (\mathrm{Pic}_{X/k}^{0})_{red}(k)$.
Now
\begin{align*}
\mathrm{Pic}(X)/f^{*}\mathrm{Pic}(C) &= (\mathrm{Pic}_{X/k})_{red}(k)/\mathrm{Pic}_{C/k}(k) \\
(\mathrm{Pic}_{X/k})_{red}(k)/(\mathrm{Pic}_{X/k}^{0})_{red}(k) &= \mathrm{NS}(X) \\
(\mathrm{Pic}_{C/k})_{red}(k)/(\mathrm{Pic}_{C/k}^{0})_{red}(k) &= \mathbb{Z}.
\end{align*}
Therefore,
\[ \mathrm{Pic}(X)/f^{*}\mathrm{Pic}(C) \cong \mathrm{NS}(X)/\mathbb{Z}. \]
By Proposition \ref{Neron severi}, $\mathrm{NS}(X)$ is f.g. of rank $\rho(X)$.
Thus, the group $\mathrm{Pic}(X)/f^{*}\mathrm{Pic}(C)$ is f.g. of rank $\rho(X) - 1$.
Since $\mathrm{Pic}(X)_{0} = \mathrm{Ker}(\mathrm{deg} \circ r_{\eta} : \mathrm{Pic}(X) \rightarrow \mathbb{Z})$, one has
$\mathrm{Pic}(X)/\mathrm{Pic}(X)_{0} \cong \mathrm{Im}(\mathrm{deg} \circ r_{\eta}) \cong \mathbb{Z}$. Hence
\begin{equation} \label{rank of rho-2}
 \text{the group $\mathrm{Pic}(X)_{0}/f^{*}\mathrm{Pic}(C)$ is f.g. of rank $\rho(X) - 2$}  
 \end{equation}
\indent On the other hand, 
\begin{equation} \label{rank of Jac}
\mathrm{Jac}(X_{\eta})(K) = \frac{\mathrm{Pic}(X)_{0}}{\mathrm{Pic}(X)_{f}} 
                                         =\frac{\mathrm{Pic}(X)_{0}/f^{*}\mathrm{Pic}(C)}
                                         {\mathrm{Pic}(X)_{f}/f^{*}\mathrm{Pic}(C)}
                                         = \frac{\mathrm{Pic}(X)_{0}/f^{*}\mathrm{Pic}(C)}{\oplus_{c}((\mathrm{Num}_{c}(X)/\mathbb{Z}[\overline{X_{c}}]) \oplus \mathbb{Z}/m_{c}\mathbb{Z})}
\end{equation}
Here, the first quality uses Proposition \ref{prepare Mordell-Weil Theorem for function fields} (iii), 
and the third Proposition \ref{prepare Mordell-Weil Theorem for function fields} (i).\\
By $(\ref{rank of rho-2})$ and $(\ref{rank of Jac})$, 
\[ \text{the group $\mathrm{Jac}(X_{\eta})(K)$ is f.g. of rank 
$\rho(X) - 2 - \sum_{c \in C}(\mathrm{\# \mathrm{Irr}}(X_{c}) - 1)$}. \]
In particular, we see that 
\begin{equation} \label{A=0}
 \text{``$A = 0$ \ $\Rightarrow$ \ $\mathrm{Jac}(X_{\eta})(K)$ is f.g.''}
 \end{equation}
\indent Next, assume $A \neq 0$.
Then $A \times_{\mathrm{Spec}(k)} \mathrm{Pic}_{C/k}^{0} \sim_{isog} \mathrm{Pic}_{X/k}^{0}$.
By the base extension $K/k$, we get a morphism of group $K$-varieties
\[ (A \times_{\mathrm{Spec}(k)} \mathrm{Spec}(K)) \times_{\mathrm{Spec}(K)} (\mathrm{Pic}_{C/k}^{0} \times_{\mathrm{Spec}(k)} \mathrm{Spec}(K))  \rightarrow \mathrm{Pic}_{X/k}^{0} \times_{\mathrm{Spec}(k)} \mathrm{Spec}(K). \]
Then, there is a morphism of group K-varieties
\[ \phi : A \times_{\mathrm{Spec}(k)} \mathrm{Spec}(K) \rightarrow \mathrm{Pic}_{X/k}^{0} \times_{\mathrm{Spec}(k)} \mathrm{Spec}(K) \cong \mathrm{Pic}_{X \times_{\mathrm{Spec}(k)} \mathrm{Spec}(K)/K}^{0}. \]
Here, let $i  : X_{\eta} = X \times_{C} \mathrm{Spec}(K) \rightarrow X \times_{\mathrm{Spec}(k)} \mathrm{Spec}(K)$ be the closed $K$-immersion. 
Then, we get a pull-back of group K-varieties  
\[ i^{*} : \mathrm{Pic}^{0}_{X \times_{\mathrm{Spec}(k)} \mathrm{Spec}(K)} \rightarrow \mathrm{Pic}^{0}_{X_{\eta}/K} = \mathrm{Jac}(X_{\eta}). \]
Hence, we obtain a morphism of group K-varieties
\[ \psi : = i^{*} \circ \phi : A_{K} = A \times_{\mathrm{Spec}(k)} \mathrm{Spec}(K) \rightarrow \mathrm{Jac}(X_{\eta}). \]
Here we need:
\begin{lem} \label{finite kernel}
$\mathrm{Ker}(\psi)$ is a finite group scheme. 
\end{lem}
\begin{proof}
Assume $\mathrm{dim}(\mathrm{Ker}(\psi)) > 0$.
Then $\mathrm{Ker}(\psi)^{0} \subset A_{K}$ is an abelian $K$-variety of positive dimension.
Since $C$ is geometrically connected, $K/k$ is a primary field extension.
By \cite[Thm.~5,\ p.26]{Lang}, there is an abelian $k$-variety $B \neq 0$ such that 
\begin{equation} \label{B}
\mathrm{Ker}(\psi)^{0} \cong B_{K}.
\end{equation}
Since the m-torsion subgroup $B(k)[m]$ is isomorphic to 
$(\mathbb{Z}/m\mathbb{Z})^{2\mathrm{dim}(B)}$ for any integer $m$ which is not divisible 
by $\mathrm{char}(k)$ (cf. \cite[Prop.~(3), p.64]{Mumford A}), we see that 
\begin{equation} \label{ker is not f.g}
\text{$B(k)$ is not f.g.}
\end{equation}
\indent On the other hand, since $\mathrm{Ker}(\psi)^{0}(K) = B(K)$ by $(\ref{B})$, we have
\[ B(k) = \mathrm{Ker}(\psi)^{0}(K) \cap B(k) \subset \mathrm{Ker}(\psi)(K) \cap A(k). \]
Since $\mathrm{Jac}(X_{\eta})(K) = \mathrm{Pic}(X)_{0}/\mathrm{Pic}(X)_{f}$ by Proposition \ref{prepare Mordell-Weil Theorem for function fields} (iii), we have 
$\mathrm{Ker}(\psi)(K) = A(K) \cap \mathrm{Pic}(X)_{f}$, so we get
\begin{equation} \label{bound of B}
B(k) \subset \mathrm{Ker}(\psi)(K) \cap A(k) = A(k) \cap \mathrm{Pic}(X)_{f}. 
\end{equation}
Thus, it suffices to show that $A(k) \cap \mathrm{Pic}(X)_{f}$ is f.g.
Indeed, by Proposition \ref{prepare Mordell-Weil Theorem for function fields} (ii),
$\mathrm{Pic}(X)_{f}/f^{*}\mathrm{Pic}^{0}(C)$ is f.g.
By (\ref{isogeny for A}), $A(k) \cap f^{*}\mathrm{Pic}^{0}(C)$ is finite. By the inclusion
\[ \frac{A(k) \cap \mathrm{Pic}(X)_{f}}{A(k) \cap f^{*}\mathrm{Pic}^{0}(C)} \subset \frac{\mathrm{Pic}(X)_{f}}{f^{*}\mathrm{Pic}^{0}(C)}, \]
we see that $A(k) \cap \mathrm{Pic}(X)_{f}$ is f.g.
By $(\ref{bound of B})$, we see that 
\begin{equation} \label{ker is f.g}
\text{$B(k)$ is f.g.}
\end{equation}
By $(\ref{ker is not f.g})$ and $(\ref{ker is f.g})$, we obtain a contradiction. Hence $\mathrm{Ker}(\psi)$ is a finite group scheme.
\end{proof}
Let us go back to the proof of Theorem \ref{Mordell-Weil Theorem for function fields}.
The morphism $\psi : A_{K} \rightarrow \mathrm{Jac}(X_{\eta})$ is either constant or surjective because $\mathrm{dim}(\mathrm{Jac}(X_{\eta})) = 1$.
Since $\mathrm{Ker}(\psi)$ is finite by Lemma \ref{finite kernel}, $\psi$ is surjective, so $\psi$ is isogeny.
In particular,
\begin{equation} \label{A not = 0}
\text{``$A \neq 0$ \ $\Rightarrow$ \ 
$\mathrm{Jac}(X_{\eta}) \sim_{isog} A_{K}$ \ $\Rightarrow$ 
$\mathrm{Jac}(X_{\eta})(K)$ is not f.g.''}
\end{equation}
\indent Combining $(\ref{A=0})$ and $(\ref{A not = 0})$, we get $\mathrm{dim}(A) \leq 1$ and 
\[
\text{``$A = 0$ \ $\Leftrightarrow$ \ $\mathrm{Jac}(X_{\eta})(K)$ is f.g.'' \ \ \ and \ \ \ 
``$A \neq 0$ \ $\Leftrightarrow$ \ $\mathrm{Jac}(X_{\eta})(K)$ is not f.g.''} 
\]
Now, by $(\ref{isogeny for A})$ and $\mathrm{dim}((\mathrm{Pic}_{-/k}^{0})_{red}) = 1/2 \cdot b_{1}(-)$ (\cite[Prop. 0.7.4,\ p.69]{Cossec and Dolgachev}), we get
\[ \mathrm{dim}(A) = \mathrm{dim}((\mathrm{Pic}_{X/k}^{0})_{red}) - \mathrm{dim}((\mathrm{Pic}_{C/k}^{0})) = 1/2(b_{1}(X) - b_{1}(C)) \leq 1. \]
In particular, we get 
\[ \text{``$A = 0$  $\Leftrightarrow$  $b_{1}(X) = b_{1}(C)$'' \ \ \ and \ \ \ ``$A \neq 0$ $\Leftrightarrow$  $b_{1}(X) = b_{1}(C) + 2$''} \]
Thus, we see that $A$ satisfies assertions (i)-(iii).
In particular, if $A \neq 0$, then $\mathrm{Jac}(X_{\eta})$ is an elliptic curve, so $f$ is elliptic. On the contrary, $f$ is quasi-elliptic, then $A = 0$,
so $\mathrm{Jac}(X_{\eta})(K)$ is finitely-generated.\\

\indent Finally, we prove the existence of a $k$-isogeny 
$T \overset{isog}\sim A$. Assume $T = 0$.
Then the map $\tau : T_{K} \rightarrow \mathrm{Jac}(X_{\eta})$ is not a $K$-isogeny since $\mathrm{dim}(T_{K}) < \mathrm{dim}(\mathrm{Jac}(X_{\eta}))$,
so we get $A = 0$ by assertion (iii) and the universal property of $\tau$, and hence 
$T \overset{isog}\sim A.$\\
\indent Next, assume $T \neq 0$. Then the map $\tau : T_{K} \rightarrow \mathrm{Jac}(X_{\eta})$ is a $K$-isogeny 
since $\mathrm{dim}(T_{K}) = \mathrm{dim}(\mathrm{Jac}(X_{\eta}))$,
so $\mathrm{Jac}(X_{\eta})(K)$ is not finitely-generated.
By assertion (iii), we have a $K$-isogeny $\psi : A_{K} \rightarrow \mathrm{Jac}(X_{\eta})$, so we have a $K$-isogeny $T_{K} \overset{isog}\sim A_{K}$. Now $K/k$ is primary.
By \cite[Thm.~5,\ p.26]{Lang}, there is a $k$-isogeny of elliptic curves 
\[ T \overset{isog}\sim A. \]
Since $A$ satisfies assertions (i)-(iii), so is $T$. 
\end{proof}

\subsection{Sections}
Let $f : X \rightarrow C$ be a fibration from a surface to a curve.
\begin{definition}
A morphism $s : C \rightarrow X$ is a \textit{section} of $f$ if $f \circ s = \mathrm{id}_{C}$.
\end{definition}
We identify a section $s$ of $f$ with its image in $X$.
This is a curve $S$ such that  $f|_{S}$ is an isomorphism, or equivalently 
$(S \cdot X_{c}) = 1$ for every closed fiber $X_{c}$ of $f$.
Let $X(C)$ denote the set of sections of $f$.
For any morphism $g : Z \rightarrow S$ of regular schemes, let 
\[ Z^{\#} : = \{ \ z \in Z \ | \ \text{$g$ is smooth at $z$} \ \}. \]
We will use the following properties of sections:
\begin{lem} \label{section} 
Let $f : X \rightarrow C$ be a fibration from a surface to a curve.
Let $X_{\eta}$ be the generic fiber of $f$.
Let $K$ be the function field of $C$. Then
\begin{enumerate}
\item Giving a section $s$ of $f$ is equivalent to giving a $K$-rational point of $X_{\eta}$.
Moreover, the image of $s$ lies  in $X^{\#}$ and corresponds to 
a point of $X_{\eta}^{\#}(K)$.
\item If $f$ has a section, then $f$ has no multiple fibers.
\end{enumerate}
\end{lem}
\begin{proof}
(i) Let $s$ be a section of $f$.
By the base change $\mathrm{Spec}(K) \rightarrow C$, we get
$s \times \mathrm{id}_{K} : \mathrm{Spec}(K) \cong C \times_{C} \mathrm{Spec}(K) 
\hookrightarrow X_{\eta}$, so $s \times \mathrm{id}_{K} \in X_{\eta}(K)$.
Conversely, let $\xi \in X_{\eta}(K)$.
Take the closure $S$ of $\xi$ in $X$.
Since $k(S) \cong k(C) \cong K$ and $C$ is a normal $k$-curve, 
the proper birational map $f|_{S} : S \dashrightarrow C$ is an isomorphism, 
so $f|_{S}^{-1} \in X(C)$.\\
\indent So, it remains to show 
for any $s \in X(C)$, any closed fiber $X_{c}$ is smooth at $x : = s(c)$.
Let $R : = \mathcal{O}_{C, c}$, $m$ its maximal ideal, and $f' : X' \rightarrow C'$ the base change of $f$ by 
$\mathrm{Spec}(R) \rightarrow C$.
By $X_{c} \cong X'_{c}$, we may assume $C = \mathrm{Spec}(R)$.
Let $A : = \mathcal{O}_{X, x}$ and $m_{A}$ its maximal ideal.
Then $\mathcal{O}_{X_{c}, x} = A/mA$ and its maximal ideal is $M : = m_{A}/mA$.
The image of $m$ under the local homomorphism $R \rightarrow A \rightarrow R$ defined by $f$ and $s$ is equal to $m$.
This implies that $mA$ is not contained in $m_{A}^{2}$.
Thus the surjection 
$m_{A}/m_{A}^{2} \rightarrow m_{A}/(m_{A}^{2} + mA) = M/M^{2}$ is not injective, so 
$\mathrm{dim}(M/M^{2}) < \mathrm{dim}(m_{A}/m_{A}^{2})$.
As $X$ is smooth at $x$, $\mathrm{dim}(m_{A}/m_{A}^{2}) = \mathrm{dim}(X) = 2$.
Also $\mathrm{dim}(A/mA) = \mathrm{dim}(X_{c}) = 1$.
Hence $\mathrm{dim}(A/mA) \leq \mathrm{dim}(M/M^{2})$ is an equality, so $X_{c}$ is smooth at $x$.
As the above, $s \times \mathrm{id}_{K} \in X_{\eta}^{\#}(K)$.
\\ 
\indent (ii) If $X_{c}$ is a multiple fiber of $f$, we can write $X_{c} = m_{c}D$ with $m_{c} \in \mathbb{Z}_{> 1}$.
If $S \subset X$ is a section of $f$, then $(S \cdot X_{c}) = m_{c}(S \cdot D) \geq m_{c} > 1$.
So we obtain a contradiction.
\end{proof}

Let $S$ be a Dedekind scheme and $K$ its function field. 
Let $G$ be a separated group $K$-scheme of finite type.
Let $X$ be a smooth group $S$-scheme of finite type.
\begin{definition} 
$X$ is \textit{N\'{e}ron model} of $G$ if it satisfies the following conditions$:$
\begin{enumerate}
\item $X$ is an $S$-model of $G$, i.e. $X_{K} \cong G$.
\item For each smooth $S$-scheme $Y$ and each $K$-morphism $\phi_{K} : Y_{K} \rightarrow X_{K}$,
there is a unique $S$-morphism $\psi : Y \rightarrow X$ which extends $\phi_{K}$.
\end{enumerate}
\end{definition}

\subsection{Jacobian fibrations} 
The main object of this paper is the following.
\begin{definition}
A genus $1$ fibration $f : X \rightarrow C$ is called \textit{Jacobian} 
if it admits a section, i.e. $X(C) \neq \emptyset$.
Equivalently, $X_{\eta}$ has a $\eta$-rational point by Lemma \ref{section} (i).
\end{definition}
Now we associate to every genus $1$ fibration a Jacobian genus 1 fibration that satisfies some good properties. Roughly speaking, it is easier to study $j$ than $f$:
\begin{prop} \label{Jacobian}
(\cite[Prop.~5.2.5,\ p.299]{Cossec and Dolgachev}). 
Let $f : X \rightarrow C$ be a genus $1$ fibration. Then there is a unique $($up to $C$-isomorphism$)$ Jacobian genus $1$ fibration $j : J \rightarrow C$ such that the smooth locus $J^{\#}$ is $C$-isomorphic to the N\'{e}ron model of $\mathrm{Jac}(X_{\eta})$.
In particular, $J$ satisfies the following properties$:$
\begin{enumerate}
\item $J_{\eta}^{\#} \cong \mathrm{Jac}(X_{\eta}) \neq \emptyset$, 
where $J_{\eta}^{\#}$ is the generic fiber of $j : J^{\#} \rightarrow C$,
\item  the image of any section $C \rightarrow J$ lies in $J^{\#}$, and
\item the natural map of sections $J(C) \rightarrow \mathrm{Jac}(X_{\eta})(K)$ is bijective.
\end{enumerate}
\end{prop}
\begin{proof}
First, we prove the existence of $J$.
Let $\eta \in C$ be the generic point, $X_{\eta}$ the generic fiber of $f$, 
and $J_{\eta}$ a regular compactification of $\mathrm{Jac}(X_{\eta})$ 
as in Proposition \ref{regular compactification}.
Now, we have injections $J_{\eta} \hookrightarrow \mathbb{P}^{n}_{K} \hookrightarrow \mathbb{P}^{n}_{k} \times C$.
Take schematic closure$:$
\[ J \hookrightarrow \mathbb{P}^{n}_{k} \times C. \]
Then there is a regular projective scheme $J \rightarrow C$, flat over $C$ and with the generic fiber $J_{\eta}$.
After resolving singularities of $J$, blowing down $(-1)$ curves in closed fibers, 
and replacing $J$ by it, we obtain a genus $1$ fibration $j : J \rightarrow C$.
The uniqueness of $J$ follows from the theory of minimal models.
By the same argument as in \cite[Prop.~2.15,\ p.218]{Artin} or \cite[Thm.~1.1]{Liu and Tong}, we see that 
$J^{\#}$ is $C$-isomorphic to the N\'{e}reon model of $\mathrm{Jac}(X_{\eta})$, so get (i)-(iii).
\end{proof}
We call $j : J \rightarrow C$  \textit{the Jacobian fibration of} $f : X \rightarrow C$.
First, we see:
\begin{prop} 
(\cite[Prop.~5.3.2,\ p.303]{Cossec and Dolgachev}). \label{Jacobian isomorphism}
Every Jacobian genus 1 fibration $f : X \rightarrow C$ is $C$-isomorphic to the Jacobian fibration $j : J \rightarrow C$ of $f$.
\end{prop}
\begin{proof}
By assumption, $X(C) \neq \emptyset$.
By Lemma \ref{section} (i), $X_{\eta}^{\#}(K) \neq \emptyset$.
By Proposition \ref{Jacobian},
$J_{\eta}$ is a regular compactification of $\mathrm{Jac}(X_{\eta})$.
By Proposition \ref{regular compactification} (i), we get $X_{\eta} \cong J_{\eta}$.
This induces an isomorphism of function fields $k(X) \cong k(J)$.
Then there is a birational map $J \dashrightarrow X$. 
Since $X$ is minimal and $p_{a}(X_{\eta}) \geq 1$, any birational map to $X$ is a birational morphism (cf. \cite[Coro.~3.24, p.422]{Liu}), 
so we get $X \cong J$.
\end{proof}

\begin{prop}
$($canonical bundle formula$)$.
Let $j : J \rightarrow C$ be a Jacobian genus 1 fibration.
Let $\mathcal{L}' : = R^{1}j_{*}\mathcal{O}_{J}$ be the invertible sheaf on $C$.
Then
\[ \omega_{J} \cong j^{*}(\mathcal{L}'^{-1} \otimes \omega_{C}), \ \ \mathrm{deg}(\mathcal{L}') = - \chi(\mathcal{O}_{J}). \]
\end{prop}
\begin{proof}
By Lem.~\ref{section} (ii), $j$ has no multiple fibers.
Thus, it follows from Thm.~\ref{Canonical bundle formula}.
\end{proof}

From now on, we prove relations between some invariants of $f$ and $j$, so begin with$:$

\begin{cor} \label{jacobian of jacobian}
(\cite[p.307]{Cossec and Dolgachev}).
Let $f : X \rightarrow C$ be a genus $1$ fibration and $j : J \rightarrow C$ the Jacobian fibration of $f$.
Then there is an isomorphism of group K-varieties
\[ \mathrm{Jac}(X_{\eta}) \cong \mathrm{Jac}(J_{\eta}). \]
\end{cor}
\begin{proof}
By Proposition \ref{Jacobian},
$J_{\eta}$ is a regular compactification of $\mathrm{Jac}(X_{\eta})$.
Thus, it follows from Proposition \ref{regular compactification} (iii).
\end{proof}

The following proposition plays a key role in the proof of Theorem \ref{aim of jacobian} (i).
\begin{prop} \label{Picard isog}
Let $f: X \rightarrow C$ be a genus $1$ fibration and $j : J \rightarrow C$ the Jacobian fibration of $f$.
Then there are isogenies of abelian $k$-varieties
\[ (\mathrm{Pic}_{X/k}^{0})_{red} \sim_{isog} (\mathrm{Pic}_{J/k}^{0})_{red}, \ \ \ \mathrm{Alb}_{X/k} \sim_{isog} \mathrm{Alb}_{J/k}. \]
\end{prop}
\begin{proof}
It suffices to prove the existence of a $k$-isogeny $(\mathrm{Pic}_{X/k}^{0})_{red} \sim_{isog} (\mathrm{Pic}_{J/k}^{0})_{red}$.
By Theorem \ref{Mordell-Weil Theorem for function fields},
there are abelian $k$-varieties $T$, $T'$ such that
\[ (\mathrm{Pic}_{X/k}^{0})_{red} \sim_{isog} T \times \mathrm{Pic}_{C/k}^{0},
\ \ \ \ \ \ 
(\mathrm{Pic}_{J/k}^{0})_{red} \sim_{isog} T' \times \mathrm{Pic}_{C/k}^{0}. \]
\indent First, assume $f$ is quasi-elliptic.
By the construction of $j$, $j$ is also quasi-elliptic.
By our convention in Theorem \ref{Mordell-Weil Theorem for function fields}, 
we have
\[ T = T' = 0. \]
Thus, we get an isogeny of abelian $k$-varieties
\[ (\mathrm{Pic}_{X/k}^{0})_{red} \sim_{isog} \mathrm{Pic}_{C/k}^{0} \sim_{isog} (\mathrm{Pic}_{J/k}^{0})_{red}. \]
\indent Next, assume $f$ is elliptic.
By construction, $j$ is also elliptic.
By Theorem \ref{Mordell-Weil Theorem for function fields}, 
\[ T = \mathrm{Tr}_{K/k}(\mathrm{Jac}(X_{\eta})), \ \ \ \ \ \ T' = \mathrm{Tr}_{K/k}(\mathrm{Jac}(J_{\eta})). \]
By Corollary \ref{jacobian of jacobian}, we have $\mathrm{Jac}(X_{\eta}) \cong \mathrm{Jac}(J_{\eta})$, so we get an isomorphism of $k$-abelian varieties
\[ T \cong T'. \]
Therefore, we get an isogeny of abelian $k$-varieties
\[ (\mathrm{Pic}_{X/k}^{0})_{red}  \sim_{isog} T \times_{\mathrm{Spec}(k)} \mathrm{Pic}_{C/k}^{0} \cong T' \times_{\mathrm{Spec}(k)} \mathrm{Pic}_{C/k}^{0} \sim_{isog} (\mathrm{Pic}_{J/k}^{0})_{red}. \]
\end{proof}

Using Proposition \ref{Picard isog}, we prove the first main theorem of this section$:$

\begin{thm} \label{Picard motive of Jacobian}
Let $f : X \rightarrow C$ be a genus $1$ fibration and $j : J \rightarrow C$ the Jacobian fibration of $f$.
Then there are isomorphism of Chow motives
\[ h_{1}(X) \cong h_{1}(J), \ \ \ h_{3}(X) \cong h_{3}(J). \]
\end{thm}
\begin{proof}
First, we prove $h_{1}(X) \cong h_{1}(J)$.
By Proposition \ref{Picard isog}, there is an isogeny of abelian varieties 
\[ (\mathrm{Pic}_{X/k}^{0})_{red} \sim_{isog} (\mathrm{Pic}_{J/k}^{0})_{red}. \]
Since isogenous abelian varieties have isomorphic Chow motives (cf. \cite[Thm.~3.1]{DM}),
we get an isomorphism
\begin{equation} \label{isomorphism of h1 of Pic}
h_{1}((\mathrm{Pic}_{X/k}^{0})_{red}) \cong h_{1}((\mathrm{Pic}_{J/k}^{0})_{red}).
\end{equation}
On the other hand, by Proposition \ref{Picard motive}, we have an isomorphism of Chow motives
\begin{equation} \label{property of h1}
h_{1}(X) \cong  h_{1}((\mathrm{Pic}_{X/k}^{0})_{red}). 
\end{equation}
Combinnig $(\ref{isomorphism of h1 of Pic})$ and $(\ref{property of h1})$, we get $h_{1}(X) \cong h_{1}(J)$.
Similarly, we get $h_{3}(X) \cong h_{3}(J)$
(because $(\mathrm{Pic}_{X/k}^{0})_{red} \sim_{isog} \mathrm{Alb}_{X/k}$).
\end{proof}
\begin{rem} \label{b1}
Taking the \'etale realizations of $h_{i}(X) \cong h_{i}(J)$ ($i=1, 3$), we have
isomorphisms $H^{i}_{\textit{\'et}}(X, \mathbb{Q}_{l}) \cong H^{i}_{\textit{\'et}}(J, \mathbb{Q}_{l})$, so get $b_{i}(X) = b_{i}(J)$ ($i=1, 3$). 
\end{rem}

To prove Thm.~\ref{aim of jacobian} (ii), we recall some known relations between $f$ and $j$:

\begin{prop} \label{euler}
(\cite[Prop.~5.3.6,\ p.308]{Cossec and Dolgachev}). 
Let $f : X \rightarrow C$ be a genus $1$ fibration and $j : J \rightarrow C$ the Jacobian fibration of $f$. 
Then
\begin{enumerate}
\item $\chi(\mathcal{O}_{X}) = \chi(\mathcal{O}_{J})$ and $e(X) = e(J)$.
\item $\rho(X) = \rho(J)$ and  $b_{i}(X) = b_{i}(J)$ for every  $i \geq 0$.
\end{enumerate}
\end{prop}
\begin{proof}
(i) See \cite[Prop.~5.3.6,\ p.308]{Cossec and Dolgachev}. The heart is \cite[Thm.~4.2,\ p.482]{Liu and Lorenzini and Raynaud}.\footnote{
Liu-Lorenzini-Raynaud showed that the singular fibers $X_{t}$ and $J_{t}$ are of the ``same type'' (\cite[Thm.~6.6,\ p.455]{Liu and Lorenzini and Raynaud}). This is one of the most profound results about the relations between $f$ and $j$.}\\
(ii) See \cite[Coro.~5.3.5,\ p.310]{Cossec and Dolgachev}. The proof of (ii) is used by (i). 
\end{proof}

We prove the second main theorem of this section$:$

\begin{thm} \label{alg motive}
Let $f : X \rightarrow C$ be a genus $1$ fibration and $j : J \rightarrow C$ the Jacobian fibration of $f$.
Then there is an isomorphism of Chow motives
\[ h_{2}^{alg}(X) \cong h_{2}^{alg}(J).\]
\end{thm}
\begin{proof}
For a surface $S$, $h_{2}^{alg}(S) \cong \mathbb{L}^{\rho(S)}$.
By Proposition \ref{euler} (ii), $\rho(X) = \rho(J)$, so 
\[ h_{2}^{alg}(X) \cong \mathbb{L}^{\oplus \rho(X)} = \mathbb{L}^{\oplus \rho(J)} \cong h_{2}^{alg}(J).\]
\end{proof}

\section{Chow motives of genus one fibrations}
In this section, we prove the first main theorem of this paper:
\begin{thm} \label{motive of genus one}
Let $k$ be an algebraically closed field of arbitrary characteristic.\\
Let $f : X \rightarrow C$ be a minimal genus $1$ fibration from a smooth projective surface over $k$, and $j : J \rightarrow C$ the Jacobian fibration of $f$.
Then there is an isomorphism
\[ h(X) \cong h(J) \]
in the category $\mathrm{CH}\mathcal{M}(k, \mathbb{Q})$ of Chow motives.
\end{thm}

\textbf{Proof of Theorem \ref{motive of genus one}.}
Let us consider the CK-decompositions of Murre (Proposition \ref{Picard motive}) and Kahn-Murre-Pedrini (Proposition \ref{def of trans}) of $h(X)$ and $h(J)$, respectively$:$
\begin{align*}
h(X) &\cong 1 \oplus h_{1}(X) \oplus  h_{2}^{alg}(X) \oplus t_{2}(X) \oplus h_{3}(X) \oplus 
(\mathbb{L} \otimes \mathbb{L}).\\
h(J) &\cong 1 \oplus h_{1}(J) \oplus  h_{2}^{alg}(J) \oplus t_{2}(J) \oplus h_{3}(J) \oplus
(\mathbb{L} \otimes \mathbb{L}).
\end{align*}
By Theorem \ref{Picard motive of Jacobian} and Theorem \ref{alg motive}, we have 
\[ h_{1}(X) \cong h_{1}(J), \ \ \ h_{3}(J) \cong h_{3}(J), \ \ \ h_{2}^{alg}(X) \cong h_{2}^{alg}(J). \]
\indent Therefore, to prove Theorem \ref{motive of genus one}, it suffices to prove the following$:$
\begin{thm}\label{tmotive of genus one}
Let $f : X \rightarrow C$ be a minimal genus $1$ fibration 
and $j : J \rightarrow C$ the Jacobian fibration of $f$.
There is an isomorphism of transcendental motives
\[ t_{2}(X) \cong t_{2}(J) \ \ \ \text{in} \ \ \ \mathrm{CH}\mathcal{M}(k, \mathbb{Q}). \]
\end{thm}
\textbf{Proof of Theorem \ref{tmotive of genus one}.}
By definition, $f$ is either elliptic or quasi-elliptic.
\subsection{The transcendental motives of the quasi-elliptic surfaces} \ \indent \\
\indent First, we assume $f$ is quasi-elliptic.
Then, $j$ is also quasi-elliptic by the construction of $j$.
Let us recall the result of the author$:$
\begin{thm} (\cite{Kawabe}). \label{tmotive of qusi-elliptic}
Let $f : X \rightarrow C$ be a quasi-elliptic surface. Then $t_{2}(X) = 0$.
\end{thm} 
We apply Theorem \ref{tmotive of qusi-elliptic} to $f$ and $j$,  and get
\[ t_{2}(X) = 0 = t_{2}(J). \]
Thus, we complete the proof of Theorem \ref{tmotive of genus one}
for the case where $f$ is quasi-elliptic.
\subsection{The transcendental motives of the elliptic surfaces} \ \indent \\
\indent Next, we assume $f$ is elliptic.
By Proposition \ref{isomorphism of transcendental motives}, it suffices to prove$:$
there are elements $[\alpha] \in \mathrm{CH}_{2}(X \times J)/\mathrm{CH}_{2}(X \times J)_{\equiv}$,
$[\beta] \in \mathrm{CH}_{2}(J \times X)/\mathrm{CH}_{2}(J \times X)_{\equiv}$ such that
\[
\begin{cases}
[\alpha] \circ_{k} [\beta] = [\Delta_{J}]
& \text{{in} \ \ $\mathrm{CH}_{2}(J \times J)/\mathrm{CH}_{2}(J \times J)_{\equiv}$} \\
[\beta] \circ_{k} [\alpha] = [\Delta_{X}]
&\text{in \ \ $\mathrm{CH}_{2}(X \times X)/\mathrm{CH}_{2}(X \times X)_{\equiv}$} 
\end{cases}
\]
\indent \textbf{Step 1. Construct correspondences on the elliptic surfaces.}\\
\indent We construct the correspondences $[\alpha]$ and $[\beta]$.
Let $\eta \in C$ be the generic point.\\
Let $X_{\eta}$ and $J_{\eta}$ be the generic fibers of $f : X \rightarrow C$ and $j : J \rightarrow C$, respectively.
Since $f$ is elliptic, $X_{\eta}$ is a smooth genus $1$ curve.
By the construction of $J_{\eta}$, $J_{\eta} \cong \mathrm{Jac}(X_{\eta})$.
By Proposition \ref{genus one curve}, $X_{\eta}$ is a torsor for $J_{\eta}$.
By Theorem \ref{motive of smooth genus one}, we get
\[ h(X_{\eta}) \cong h(J_{\eta}) \ \ \ \text{in} \ \ \ \mathrm{CH}\mathcal{M}(\eta, \mathbb{Q}). \]
Thus, there are elements $a \in \mathrm{CH}_{1}(X_{\eta} \times_{\eta} J_{\eta})$ and 
$b \in \mathrm{CH}_{1}(J_{\eta} \times_{\eta} X_{\eta})$ such that
\[ 
\begin{cases}
a \circ_{\eta} b = \Delta_{J_{\eta}}  &  \text{in \ \ $\mathrm{CH}_{1}(J_{\eta} \times_{\eta} J_{\eta})$} \\
b \circ_{\eta} a = \Delta_{X_{\eta}}  &  \text{in \ \ $\mathrm{CH}_{1}(X_{\eta} \times_{\eta}  X_{\eta})$}
\end{cases}
\]
\indent In this proof, we let 
\[ r_{XJ} : \mathrm{CH}_{2}(X \times_{C} J) \rightarrow \mathrm{CH}_{1}(X_{\eta} \times_{\eta} J_{\eta}) \]
be the flat-pullback, and similar for $r_{JX}$ and $r_{JJ}$.
Let 
\[ \iota_{XJ} : \mathrm{CH}_{2}(X \times_{C} J) \rightarrow \mathrm{CH}_{2}(X \times J) \]
be the proper-pushfoward, and similar for $\iota_{JX}$ and $\iota_{JJ}$.

\begin{lem} \label{two funcotors} There are homomorphisms of groups
\[ \mathrm{CH}_{1}(X_{\eta} \times_{\eta} J_{\eta}) 
\overset{\cong}{\underset{r_{XJ}}{\longleftarrow}} 
\frac{\mathrm{CH}_{2}(X \times_{C} J)}{\oplus_{c}\mathrm{CH}_{2}(X_{c} \times_{c} J_{c})}
\underset{\iota_{XJ}}\longrightarrow \frac{\mathrm{CH}_{2}(X \times J)}{\mathrm{CH}_{2}(X \times J)_{\equiv}}. \]
\end{lem}
\begin{proof}
The left isomorphism $r_{XJ}$ follows from the local exact sequences
\[ \oplus_{c}\mathrm{CH}_{2}(X_{c} \times_{c} J_{c}) \longrightarrow \mathrm{CH}_{2}(X \times J) 
\underset{r_{XJ}}\longrightarrow \mathrm{CH}_{1}(X_{\eta} \times_{\eta} J_{\eta}) \rightarrow 0. \]
Thus, it remains to prove  $\iota_{XJ}(\oplus_{c}\mathrm{CH}_{2}(X_{c} \times_{c} J_{c})) \subset \mathrm{CH}_{2}(X \times J)_{\equiv}.$
Indeed, let $z \in \oplus_{c}\mathrm{CH}_{2}(X_{c} \times_{c} J_{c})$.
Then $z = \sum_{c}\sum_{i,j}n_{c, i, j}[E_{c, i} \times_{k(c)} F_{c, j}]$ with $n_{c, i, j} \in \mathbb{Q}$.
Here, $E_{c, i}$ and $F_{c, j}$ runs over all irreducible components of $X_{c}$ and $J_{c}$, respectively.
Thus 
\[ \iota_{XJ}(z) = \sum_{c}\sum_{i,j}n_{c, i, j}[E_{c, i} \times_{k} F_{c, j}]\in \mathrm{CH}_{2}(X \times_{C} J)_{\equiv}. \]
\end{proof}
By Lemma \ref{two funcotors}, we can define
\[  [\alpha] : =  \iota_{XJ}(r_{XJ}^{-1}(a)) 
\in \mathrm{CH}_{2}(X \times J)/\mathrm{CH}_{2}(X \times J, \mathbb{Q})_{\equiv}. \]
Similarly, we define 
\[ [\beta] : =  \iota_{JX}(r_{JX}^{-1}(b)) 
\in \mathrm{CH}_{2}(J \times X)/\mathrm{CH}_{2}(J \times X)_{\equiv}. \]

\textbf{Step 2. Calculate the correspondences on the elliptic surfaces.} \\
\indent To compute the correspondences $[\alpha] \circ_{k} [\beta]$ and $[\beta] \circ_{k} [\alpha]$, 
we prove the following:
\begin{lem} \label{lemma for main prop} Use above notations. Then
\begin{enumerate}
\item For $[y] \in \mathrm{CH}_{2}(X \times_{C} J)/\oplus_{c}\mathrm{CH}_{2}(X_{c} \times_{c} J_{c})$, 
$[z] \in  \mathrm{CH}_{2}(J \times_{C} X)/\oplus_{c}\mathrm{CH}_{2}(J_{c} \times_{c} X_{c})$,
\[ \iota_{XJ}([y]) \circ_{k} \iota_{JX}([z]) = \iota_{JJ}([y] \circ_{C} [z]) \ \ \ \text{in} \ \ \
\mathrm{CH}_{2}(J \times J)/\mathrm{CH}_{2}(J \times J)_{\equiv}. \]
\item For $d \in \mathrm{CH}_{1}(X_{\eta} \times_{\eta} J_{\eta})$, 
$e \in \mathrm{CH}_{1}(J_{\eta} \times_{\eta} X_{\eta})$,
\[r_{XJ}^{-1}(d) \circ_{C} r_{JX}^{-1}(e) = r_{JJ}^{-1}(d \circ_{\eta} e) \ \ \ \text{in} \ \ \
\mathrm{CH}_{2}(J \times_{C} J)/\oplus_{c}\mathrm{CH}_{2}(J_{c} \times_{c} J_{c}). \]
\end{enumerate}
\end{lem}
\begin{proof}
We prove (i).
By Proposition \ref{bilinear}, there is a bilinear homomorphism
\begin{align} \label{mod correspondences 1}
\circ : \frac{\mathrm{CH}_{2}(J \times X)}{\mathrm{CH}_{2}(J \times X)_{\equiv}} \times 
\frac{\mathrm{CH}_{2}(X \times J)}{\mathrm{CH}_{2}(X \times J)_{\equiv}} &\rightarrow 
\frac{\mathrm{CH}_{2}(J \times J)}{\mathrm{CH}_{2}(J \times J)_{\equiv}}  \\
([\delta], [\gamma]) &\mapsto [ \gamma] \circ_{k} [\delta] : = [ \gamma \circ_{k} \delta] \notag
\end{align}
Similarly, there is a bilinear homomorphism
\begin{align} \label{mod correspondences 2}
\circ : \frac{\mathrm{CH}_{2}(J \times_{C} X)}{\oplus_{c}\mathrm{CH}_{2}(J_{c} \times X_{c})} \times 
\frac{\mathrm{CH}_{2}(X \times_{C} J)}{\oplus_{c}\mathrm{CH}_{2}(X_{c} \times J_{c})} &\rightarrow 
\frac{\mathrm{CH}_{2}(J \times_{C} J)}{\oplus_{c}\mathrm{CH}_{2}(J_{c} \times J_{c})}  \\
([z], [y]) &\mapsto [y] \circ_{C} [z] : = [y \circ_{C} z] \notag
\end{align}
Thus, in $\mathrm{CH}_{2}(J \times_{C} J)/\oplus_{c}\mathrm{CH}_{2}(J_{c} \times_{c} J_{c})$,
\begin{align*}
 \iota_{XJ}([y]) \circ_{k} \iota_{JX}([z]) &= [\iota_{XJ}(y)] \circ_{k} [\iota_{JX}(z)] 
\overset{(\ref{mod correspondences 1})}= [\iota_{XJ}(y) \circ_{k} \iota_{JJ}(z)] \\
&= [\iota_{JJ}(y \circ_{C} z)]  = \iota_{JJ}([y \circ_{C} z]) 
\overset{(\ref{mod correspondences 2})}= \iota_{JJ}([y] \circ_{C} [z]).
\end{align*}   
Here, the third equality uses Lemma \ref{correspondences and closed immersion} for $B' = C$ and $B = \mathrm{Spec}(k)$. So, we get (i).
The proof of (ii) is similar to (i) 
(Use Lemma \ref{correspondences and base change} for $U = \mathrm{Spec}(k(C))$ and $B = C$).          
\end{proof}

Now, we prove the following main proposition in this proof:

\begin{prop} \label{Calculate correspondences} Use above notations.
Then in $\mathrm{CH}_{2}(J \times J)/\mathrm{CH}_{2}(J \times J)_{\equiv}$, 
\[ [\alpha] \circ_{k} [\beta] = [\Delta_{J}]. \]
\end{prop}
\begin{proof} 
In $\mathrm{CH}_{2}(J \times J)/\mathrm{CH}_{2}(J \times J)_{\equiv}$, 
\begin{align*}
[\alpha] \circ_{k} [\beta] &= \iota_{XJ}(r_{XJ}^{-1}(a)) \circ_{k} \iota_{JX}(r_{JX}^{-1}(b)) && \text{by definition} \\
&= \iota_{JJ}(r_{XJ}^{-1}(a) \circ_{C} r_{JX}^{-1}(b)) && \text{by Lemma \ref{lemma for main prop} (i)} \\
&= \iota_{JJ}(r_{JJ}^{-1}(a \circ_{\eta} b)) && \text{by Lemma \ref{lemma for main prop} (ii)} \\
&= \iota_{JJ}(r_{JJ}^{-1}(\Delta_{J_{\eta}})) && \text{by $a \circ_{\eta} b = \Delta_{J_{\eta}}$} \\
&= [\Delta_{J}]. 
\end{align*}
\end{proof}
\indent By Proposition \ref{Calculate correspondences}, we get
$[\alpha] \circ_{k} [\beta] = [\Delta_{J}]$ in 
$\mathrm{CH}_{2}(J \times J)/\mathrm{CH}_{2}(J \times J)_{\equiv}$.\\
Similarly, $[\beta] \circ_{k} [\alpha] = [\Delta_{X}]$ in $\mathrm{CH}_{2}(X \times X)/
\mathrm{CH}_{2}(X \times X)_{\equiv}$.
Thus, we get 
\[ t_{2}(X) \cong t_{2}(J) \ \ \ \text{in} \ \ \ \mathrm{CH}\mathcal{M}(k, \mathbb{Q}). \]
Therefore, we complete the proof of Theorem \ref{tmotive of genus one} for the case where $f$ is elliptic.
Therefore, we complete the proof of Theorem \ref{tmotive of genus one}, 
and hence of Theorem \ref{motive of genus one}.
\begin{rem}
Taking the \'etale realization of $t_{2}(X) \cong t_{2}(J)$, we have
an isomorphism $H^{2}_{\textit{\'et}}(X, \mathbb{Q}_{l})_{tr} \cong H^{2}_{\textit{\'et}}(J, \mathbb{Q}_{l})_{tr}$,  so get $h^{2}_{\textit{\'et}}(X, \mathbb{Q}_{l})_{tr} = h^{2}_{\textit{\'et}}(J, \mathbb{Q}_{l})_{tr}$.
\end{rem}

\begin{rem}
Thm.~\ref{tmotive of genus one} is proved
for the case where $X$ is an Enriques surface
with an elliptic fibration \cite{Coombes} (see Prop.~\ref{Coombes's result}).
Then we have $b_{2}(X) = 10$ by Def.~\ref{def of Enriques}, so
$b_{1}(X) = 0$ by Prop.~\ref{list},
and hence $(\mathrm{Pic}_{X/k}^{0})_{red} =  0$ by Prop.~\ref{Neron severi} (i).
Thus, we get $h_{1}(X) = h_{3}(X) = 0$
by the argument as in the proof of Thm.~\ref{Picard motive of Jacobian}.
\end{rem}

\begin{rem}
As studied in \cite[Ch.~16]{HuyK3} the case of elliptic K3 surfaces, it is expected that there is an equivalence between bounded derived categories of coherent sheaves $\mathrm{D^{b}}(X) \cong 
\mathrm{D^{b}}(J)$.
This again induces $h(X) \cong h(J)$ (by \cite{FV21} or \cite{Huymotives}).
\end{rem}

\section{Kimura-finiteness} 
\indent Here we collect some properties of Kimura-finiteness for the reader's convenience.
\begin{definition} Let $\mathcal{C}$ be a $\mathbb{Q}$-linear pseudo-abelian tensor category
(e.g. $\mathrm{CH}\mathcal{M}(k, \mathbb{Q})$)
\begin{enumerate}
\item An object $A$ of $\mathcal{C}$ is \textit{evenly finite} if $\wedge^{n}(A) = 0$ for $n$ large enough.
\item An object $A$ of $\mathcal{C}$ is \textit{oddly finite} if $\mathrm{Sym}^{n}(A) = 0$ for $n$ large enough.
\item An object $A$ of $\mathcal{C}$ is \textit{Kimura-finite} if there is a decomposition 
$A = A_{+} \oplus A_{-}$ such that $A_{+}$ is even and $A_{-}$ is odd.
\end{enumerate}
\end{definition}

\begin{conj} (\cite{Kimura} or \cite[Ch.~12]{Andre}).
Every Chow motive is Kimura-finite.
\end{conj}
For example, the motives $1$ and $\mathbb{L}$ are Kimura-finite.
\begin{prop} \label{Properties of Kimura-finiteness}
Let $k$ be a field.
\begin{enumerate}
\item The motive of any smooth projective curve over $k$ is Kimura-finite.
\item The motive of any abelian variety is Kimura-finite.
\item Let $M$ and $N$ be Kimura-finite dimensional motives.
Then $M \oplus N$ and $M \otimes N$ are Kimura-finite.
\item Let $\pi : V \rightarrow W$ be a dominant morphism of smooth projective varieties over $k$.
If $h(V)$ is Kimura-finite, so is $h(W)$.
\item (Birational invariant) Let $X$ and $Y$ be smooth projective surfaces over $k$ which are birationally equivalent.
If $h(X)$ is Kimura-finite, so is $h(Y)$.
\end{enumerate}
\end{prop}
\begin{proof} 
See \cite{Kimura}.
In particular, (v) follows from Manin's blow-up formula \cite{Manin}.
\end{proof} 

\begin{rem} \label{algebraic motives}
We only mention the case of surfaces.
Let $k = \overline{k}$
and $S \in \mathcal{V}(k)$ a surface.
Then $h_{i}(S)$ ($i \neq 2$)  and $h_{2}^{alg}(S)$ are Kimura-finite.
Indeed,  $h_{0}(S) \cong 1$, $h_{4}(S) \cong \mathbb{L}^{\otimes 2}$, and 
$h_{2}^{alg}(S) \cong \mathbb{L}^{\oplus \rho(S)}$.
Thus, $h_{0}$, $h_{4}$, and $h_{2}^{alg}$ are Kimura-finite.
By Prop.~\ref{Picard motive}, $h_{1}(S) \cong h_{1}((\mathrm{Pic}_{S/k}^{0})_{red})$.
By Prop.~\ref{Properties of Kimura-finiteness} (ii),
$h_{1}((\mathrm{Pic}_{S/k}^{0})_{red})$ is Kimura finite, so is $h_{1}$.
Similarly, $h_{3}$ is also Kimura-finite. However, the Kimura-finiteness of $t_{2}(S)$ is unknown in general.
At least the following are known to have Kimura-finiteness:
\begin{enumerate}
\item varieties of dimension $\leq 3$ rationally dominated by products of curves (\cite[Ex. 3.15]{V}). In particular, the K3 surfaces (\cite{Par}, \cite{ILP}).
\item K3 surfaces with Picard number 19, 20 (\cite{Pedrini 12}).
\item some K3 surfaces obtained as complete intersections (\cite{Late}, \cite{BL}). 
\item many examples of surfaces with $p_{g} = 0$ over $\mathbb{C}$ 
(\cite{Guletskii and Pedrini}, \cite{PW}).
\end{enumerate}
\end{rem}
In the proof of Theorem \ref{Main 4}, we will use: 
\begin{prop} \label{finiteness for hyper surfaces} (\cite{Kimura})
Let $k$ be a field and let $C, D \in \mathcal{V}(k)$ be curves.
Let $G$ be a finite group which acts freely on $C \times D$.
Let $X$ be a surface which is birational to $C \times D/G$.
Then $h(X)$ is Kimura-finite.
\end{prop}
\begin{proof}
It follows from Proposition \ref{Properties of Kimura-finiteness}.
Indeed, both $h(C)$ and $h(D)$ are Kimura-finite by (i).
Then $h(C \times D) = h(C) \otimes h(D)$ is Kimura-finite by (iii),
so $h(C \times D/G)$ is Kimura-finite by (iv), 
and hence $h(X)$ is Kimura-finite by (v).
\end{proof}

\begin{prop} (\cite[Coro.~7.6.11,\ p.181]{Kahn and Murre and Pedrini}).
\label{equivalent finite dimensionality}
Let $S \in \mathcal{V}(\mathbb{C})$ be a surface.
Then the following properties are equivalent$:$
\begin{enumerate}
\item $a_{S} : \mathrm{CH}_{0}(S)_{\mathbb{Z}}^{0} \cong \mathrm{Alb}_{S/\mathbb{C}}(\mathbb{C});$
\item $p_{g}(S) = 0$ and $h(S)$ is Kimura-finite in $\mathrm{CH}\mathcal{M}(\mathbb{C}, \mathbb{Q})$$;$
\item $t_{2}(S) = 0$.
\end{enumerate}
\end{prop}

\section{Classification of algebraic surfaces}
Here we quickly review the classification of surfaces for the reader's convenience.\\
In this section, let $k$ be an algebraically closed field
and $S \in \mathcal{V}(k)$ a surface.
\subsection{\textbf{Kodaira dimension}}
We define the \textit{Kodaira dimension} of $S$ to be
\[ \kappa(S) : = 
\begin{cases}
- \infty & \text{if $P_{m}(S) = h^{0}(S, \omega_{S}^{\otimes m}) = 0$ for every $m \geq 1$} \\
\mathrm{tr.deg}_{k}(\oplus_{m \geq 0}H^{0}(S, \omega_{S}^{\otimes m})) - 1 & \text{otherwise}
\end{cases}
 \]
Then $\kappa = -\infty$, $0$, $1$, or $2$.
Since $P_{m}$ is a birational invariant, so is $\kappa$.
A surface $S$ is \textit{minimal} if and only if 
$S$ do contain smooth rational curves $E$ satisfying $(E^{2}) = (E \cdot K_{S}) = -1$.
If $\kappa(S) \geq 0$, then $S$ is minimal if and only if $K_{S}$ is \textit{nef}, i.e.
$(K_{S} \cdot C) \geq 0$ for every curve $C$.
We denote by $\equiv$ the numerical equivalence of divisors.
Recall:
\begin{thm} \label{kodaira classification of surfaces}
For any surface $S \in \mathcal{V}(k)$, there is a birational morphism 
$f : S \rightarrow S'$ onto a minimal surface $S' \in \mathcal{V}(k)$ that
satisfies one of the following properties:
\begin{enumerate}
\item $\kappa(S') = - \infty$, 
$S' \cong \mathbb{P}^{2}$ or $S'$ is a minimal ruled surface;
\item $\kappa(S') = 0$, $(K_{S'}^{2}) = 0$, $K_{S'} \equiv 0;$
\item $\kappa(S') = 1$, $(K_{S'}^{2}) = 0$, $K_{S'} \not\equiv 0;$
\item $\kappa(S') = 2$, $(K_{S'}^{2}) > 0.$
\end{enumerate}
In particular, if $\kappa(S) \geq 0$, $S'$ is unique.
\end{thm}

\subsection{Surfaces not of general type}  We focus on surfaces with $\kappa < 2$.

\begin{thm} (\cite[Thm.~13.2,\ p.195]{Badescu}). \label{kodaira-}
Let $S \in \mathcal{V}(k)$ be a surface. 
$\kappa(S) = - \infty$ if and only if $S$ is birationally ruled, i.e. it is birational to $\mathbb{P}^{1} \times C$ for some curve $C$.
\end{thm}

\begin{prop} \label{relation betti} (e.g. \cite[Thm.~5.1,\ p.72]{Badescu}).
Let $S \in \mathcal{V}(k)$ be a surface.
Then 
\[ 10 - 8  q + 12  p_{g} = (K_{S}^{2}) + b_{2} + 2 \ \Delta,\]
where $\Delta = 2  q - b_{1}$.
Also, $\Delta = 0$ if $\mathrm{char}(k) = 0$, and $0 \leq \Delta \leq 2  p_{g}$ 
if $\mathrm{char}(k) > 0$.
\end{prop}

Using Proposition \ref{relation betti}, we get the following$:$

\begin{prop} \label{list}
Let $S \in \mathcal{V}(k)$ be a minimal surface with $(K_{S}^{2}) = 0$ and $p_{g} \leq 1$.
\begin{enumerate}
\item $b_{2} = 22$,\ \ $b_{1} = 0$,\ \ $\chi = 2$,\ \ $q = 0$,\ \ $p_{g} = 1$,\ \ $\Delta = 0$.
\item $b_{2} = 14$,\ \ $b_{1} = 2$,\ \ $\chi = 1$,\ \ $q = 1$,\ \ $p_{g} = 1$,\ \ $\Delta = 0$.
\item $b_{2} = 10$,\ \ $b_{1} = 0$,\ \ $\chi = 1$,\ \ $q = 0$,\ \ $p_{g} = 0$,\ \ $\Delta = 0$.
\item $b_{2} = 10$,\ \ $b_{1} = 0$,\ \ $\chi = 1$,\ \ $q = 1$,\ \ $p_{g} = 1$,\ \ $\Delta = 2$.
\item $b_{2} = 6$,\ \ \ $b_{1} = 4$,\ \ $\chi = 0$,\ \ $q = 2$,\ \ $p_{g} = 1$,\ \ $\Delta = 0$.
\item $b_{2} = 2$,\ \ \ $b_{1} = 2$,\ \ $\chi = 0$,\ \ $q = 1$,\ \ $p_{g} = 0$,\ \ $\Delta = 0$.
\item $b_{2} = 2$,\ \ \ $b_{1} = 2$,\ \ $\chi = 0$,\ \ $q = 2$,\ \ $p_{g} = 1$,\ \ $\Delta = 2$.
\end{enumerate}
In particular, any minimal surface with $\kappa = 0$ belongs to the above list. 
\end{prop}

\begin{definition} \label{def of Enriques}
A minimal surface $S \in \mathcal{V}(k)$ with $\kappa(S) = 0$.\\
\indent $\bullet$ $S$ is called an \textit{Enriques surface} if $b_{2}(S) = 10$.
\end{definition}

\begin{thm}\label{Enriques has genus 1}(\cite[Thm.~10.17]{Badescu}).
Any Enriques surface has a genus $1$ fibration.
\end{thm}

\begin{prop} (e.g. \cite[Thm.~8.6,\ p.113]{Badescu}).
\label{hyper elliptic and quasi hyper-elliptic}
Let  $S \in \mathcal{V}(k)$ be a minimal surface with $\kappa(S) = 0$.
Let $\mathrm{alb}_{S} : S \rightarrow \mathrm{Alb}_{S/k}$ be the Albanese morphism of $S$.
If $b_{1}(S) = 2$, then the morphism $\mathrm{alb}_{S}$ gives rise to a fibration $a : S \rightarrow E$ 
onto an elliptic curve $E$, all of whose fibers are integral curves of arithmetic genus one.
\end{prop}

\begin{definition} \label{definition of hyper elliptic}
Let $S  \in \mathcal{V}(k)$ be a minimal surface with $\kappa(S) = 0$ and $b_{2}(S) = 2$.
Let $a : S \rightarrow E$ be the Albanese fibration of $S$
as in Proposition \ref{hyper elliptic and quasi hyper-elliptic}.
\begin{enumerate}
\item  $S$ is \textit{hyper-elliptic} if the generic fiber of $a$ is \textit{smooth}.
\item  $S$ is \textit{quasi hyper-elliptic} if the generic fiber of $a$ is \textit{non-smooth}.
\end{enumerate}
\end{definition} 

To prove the Kimura-finiteness of hyper-elliptic surfaces, we need the following:

\begin{prop} (\cite[Thm.~4,\ p.35]{Bombieri and Mumford II}). \label{hyper-elliptic}
Let $S$ be a hyper-elliptic surface. Then 
\[ S \cong E \times F/G, \]
where $E$ and $F$ are elliptic curves, and $G$ is a finite subgroup scheme of $E$. 
\end{prop}

As far as I know, the motive of the following object is not well-understood:

\begin{thm} \label{Kodaira dimension one}
(e.g. \cite[Thm.~9.9,\ p.129]{Badescu}).
Let $S \in \mathcal{V}(k)$ be a minimal surface.
If $\kappa(S) = 1$, then $S$ has a genus $1$ fibration.
\end{thm}

To prove the second main theorem, we need:

\begin{prop} \label{inequality} (e.g. \cite[Rem.~8.3,\ p.112]{Badescu}).
Let $f : X \rightarrow C$ be a minimal genus $1$ fibration.
Use the same notations as in the canonical bundle formula.
We set
\[ \lambda(f) : = 2 p_{a}(C) - 2 + \chi(\mathcal{O}_{X}) + \mathrm{length}(T) + \sum_{i=1}^{r}n_{i}/m_{i}. \]
\begin{enumerate}
\item $\lambda(f) < 0$ if and only if $\kappa(X) = - \infty$.
\item $\lambda(f) = 0$ if and only if $\kappa(X) = 0$.
\item $\lambda(f) > 0$ if and only if $\kappa(X) = 1.$
\end{enumerate}
\end{prop}

\section{Chow motives of surfaces not of general type with $p_{g} = 0$}
In this section, we prove the second main theorem of this paper (Theorem \ref{characteristic $p$}).\\
Let $k$ be an algebraically closed field and let $X \in \mathcal{V}(k)$ be a surface.\\ 
\indent Let us recall the result of Bloch-Kas-Lieberman$:$
\begin{thm} (\cite{Bloch and Kas and Lieberman}). \label{result of BKL}
Assume that $k = \mathbb{C}$, $p_{g} = 0$, and $\kappa < 2$.
Then 
\[ a_{X} : \mathrm{CH}_{0}(X)^{0}_{\mathbb{Z}} \cong \mathrm{Alb}_{X/\mathbb{C}}(\mathbb{C}).\]
\end{thm}
By Proposition \ref{equivalent finite dimensionality},
Theorem \ref{result of BKL} is equivalent to the following$:$
\begin{thm} \label{motivic vertion 2 of BKL}
Assume that $k = \mathbb{C}$, $p_{g} = 0$, and $\kappa < 2$.
Then $h(X)$ is Kimura-finite.
\end{thm}

In this paper, we generalize Theorem \ref{motivic vertion 2 of BKL} to arbitrary characteristic$:$
\begin{thm} \label{characteristic $p$}
Let $X$ be a smooth projective surface over an algebraically closed field $k$.
If $p_{g}(X) = 0$ and $\kappa(X) < 2$, then $h(X)$ is Kimura-finite in 
$\mathrm{CH}\mathcal{M}(k, \mathbb{Q})$.
\end{thm}
\textbf{Proof of Theorem \ref{characteristic $p$}.}
The ideas of the proof are based on \cite[Prop.~4, p.138]{Bloch and Kas and Lieberman}, 
\cite[Coro.~2.12,\ p.187]{Guletskii and Pedrini}, and \cite[Thm.~11.10,\ p.313]{Voisin}.\\
 
(i) we assume $\kappa(X) < 0$.
By Theorem \ref{kodaira-}, 
we have $X \sim_{birat} \mathbb{P}^{1} \times C$ for some $k$-curve $C$. 
We apply Proposition \ref{finiteness for hyper surfaces} to $G = 0$, and see that $h(X)$ is Kimura-finite.\\

So we assume $\kappa(X) \geq 0$.
By Proposition \ref{Properties of Kimura-finiteness} (v), we may assume $X$ is minimal.
Since $0 \leq \kappa < 2$, we have $(K^{2}) = 0$ by Theorem \ref{kodaira classification of surfaces}.
Then the Noether formula $10 - 8  q + 12  p_{g} = (K^{2}) + b_{2} + 2 \Delta$ becomes 
\[ 10 - 8q = b_{2}. \]
Since $b_{2} \geq 0$, we must consider the following two cases$:$\\

\indent $(a)$ $q(X) = 0$, \ $b_{2}(X) = 10;$\\
\indent $(b)$ $q(X) = 1$, \ $b_{2}(X) = 2$.\\

\begin{lem} \label{has elliptic}
Let $S \in \mathcal{V}(k)$ be a minimal surface.
If $p_{g} = 0$ and $0 \leq \kappa < 2$, then $S$ has a genus $1$ fibration.
\end{lem}
\begin{proof}
First, we assume $\kappa = 0$.
If $\kappa = p_{g} = q = 0$, then $S$ is an Enriques surface, 
so $S$ has a genus $1$ fibration by Theorem \ref{Enriques has genus 1}.
If $\kappa = p_{g} = 0$ and $q = 1$, then $b_{2} = 2$ as in (b).
By Definition \ref{definition of hyper elliptic}, $S$ is hyper-elliptic or quasi-hyper elliptic, so $S$ has a genus $1$ fibration.
Next, we assume $\kappa = 1$.
Then $S$ has a genus $1$ fibration by Theorem \ref{Kodaira dimension one}.
\end{proof}

By Lemma \ref{has elliptic}, we see that $X$ has a genus $1$ fibration
\[ f : X \rightarrow C. \]
Since $f$ is a fibration, $q(X) \geq q(C) = p_{a}(C)$. 
Since $\chi = 1 - q + p_{g} = 1 + q$, by $(a)$ and $(b)$, we must consider the following two cases$:$\\ 

\indent $(a)$ $\chi(\mathcal{O}_{X}) = 1$, \ $p_{a}(C) = 0$, \ $b_{2}(X) = 10;$\\
\indent $(b)$ $\chi(\mathcal{O}_{X}) = 0,$ \ $p_{a}(C) \leq 1$, \ $b_{2}(X) = 2$.\\

\noindent Let $j : J \rightarrow C$ be the Jacobian fibration of $f$.
By Theorem \ref{motive of genus one}, we get an isomorphism  
\[ h(X) \cong h(J) \ \ \ \text{in} \ \ \ \mathrm{CH}\mathcal{M}(k, \mathbb{Q}). \] 
Thus, it suffices to prove that $h(J)$ is Kimura-finite (In fact, we only use $t_{2}(X) \cong t_{2}(J)$).\\

\indent (ii) we assume $(a)$.
So $\chi(\mathcal{O}_{X}) = 1$ and $p_{a}(C) = 0$.\\
By Proposition \ref{euler} (i), $\chi(\mathcal{O}_{J}) = \chi(\mathcal{O}_{X}) = 1$.
Thus,
\[ \lambda(j) : = 2 p_{a}(C) - 2 + \chi(\mathcal{O}_{J}) = 0 - 2 + 1 = -1. \]
We apply Proposition \ref{inequality} to $\lambda(j) = - 1$, and get 
\[ \kappa(J) < 0. \]
By (i), we see that $h(J)$ is Kimura-finite.\\

\indent (iii) we assume $(b)$.
So $\chi(\mathcal{O}_{X}) = 0$, $p_{a}(C) \leq 1$, and $b_{2}(X) = 2$.\\
\indent (iii-i) We prove that $J$ is a hyper elliptic surface or a quasi hyper-elliptic surface.\\
\indent First, we prove $\kappa(J) \leq 0$. 
By Proposition \ref{euler} (i), $\chi(\mathcal{O}_{J}) = \chi(\mathcal{O}_{X}) = 0$.
Thus, 
\[ \lambda(j) : = 2 p_{a}(C) - 2 + \chi(\mathcal{O}_{J}) \leq 2 - 2  + 0 = 0.\]
We apply Proposition \ref{inequality} to $\lambda(j) \leq 0$, and get 
\[ \kappa(J) \leq 0. \]
By (i), we only consider the case where $\kappa(J) = 0$.\\
\indent Next, we prove $b_{2}(J) = 2$.
Now, $b_{2}(X) = 2$.
By Proposition \ref{euler} (ii), we get $b_{2}(J)  = b_{2}(X) = 2$.
By Definition \ref{definition of hyper elliptic}, $J$ is hyper-elliptic or quasi hyper-elliptic.\\
\indent (iii-ii) We prove that $M(J)$ is Kimura-finite.\\
\indent First, we assume that $J$ is quasi hyper-elliptic.
By Definition \ref{definition of hyper elliptic}, $J$ has a quasi-elliptic fibration $j' : J \rightarrow C'$.
Using Theorem \ref{tmotive of qusi-elliptic}, we have 
\[ t_{2}(J) = 0. \]
Then $h(J) \cong \oplus_{i = 0, \neq 2}^{4}h_{i}(J) \oplus h_{2}^{alg}(J)$.
By Remark \ref{algebraic motives},
$h_{i}(J) (i \neq 2)$ and $h_{2}^{alg}(J)$ are Kimura-finite.
By Proposition \ref{Properties of Kimura-finiteness} (iii),
we see that $h(J)$ is Kimura-finite.\\
\indent Next, we assume that $J$ is hyper-elliptic. By Proposition \ref{hyper-elliptic}, 
there are elliptic curves $E$, $F$, and a finite subgroup scheme $G$ of $E$ such that
\[ J \cong (E \times F)/G. \]
By Proposition \ref{finiteness for hyper surfaces}, $h(J)$ is Kimura-finite.
This completes the proof of Theorem \ref{characteristic $p$}.
\begin{rem}
Let $f : X \rightarrow C$ be a genus 1 fibration with $p_{g}(X) = 0$, and 
$j : J \rightarrow C$ the Jacobian fibration of $f$.
By the above arguments, $\kappa(X) \geq \kappa(J)$ (if $X$ is an Enriques surface, then $J$ is rational by the argument as in (ii), so $\kappa(X) = 0$ and $\kappa(J) < \infty$).
``Roughly speaking", it is easier to prove 
the Kimura-finiteness of surfaces with low Kodaira dimension than that of 
surfaces with high Kodaira dimension. 
\end{rem}
\begin{rem} Surfaces with $p_{g} = 0$$:$

\begin{center}
\begin{tabular}{|c|c|c|}
\hline
$$ & $\kappa = 0$ & \ $\kappa = 1$ \  \\
\hline
$q = 0$ & Enriques surface & canonical fibration \\
           &                          & \\
\hline
$q = 1$ & Albanese fibration & canonical fibration or Albanese fibration\\
           &                             &           \\
\hline
\end{tabular}
\end{center}
\indent Let $X$ be a surface with $p_{g} = 0$, $\kappa = 1$, and $q = 1$.
Bloch-Kas-Lieberman considered the Albanese fibration $a : X \rightarrow \mathrm{Alb}_{X/k}$.
The genus of the generic fiber of $a$ is $\geq 1$.\\
\indent On the other hand, we consider the canonical fibration $f : X \rightarrow C$.
The genus of the generic fiber of $f$ is equal to $1$.
``Roughly speaking", it is easier to treat genus 1 fibrations 
than genus $g \geq 2$ fibrations, in positive characteristics.
\end{rem}

\subsection*{Acknowledgements}
First and foremost, I would like to thank my Ph.D. supervisor, Prof.~Masaki Hanamura 
for many helpful comments and suggestions.
I am deeply grateful to Prof.~Takao Yamazaki for his comments, and 
giving me an opportunity to talk about this paper 
at Algebraic Seminar in Tohoku University.
I would like to thank Prof. Nobuo Tsuzuki for his generous support, 
and the referee for very useful comments and suggestions.
I express my appreciation for the hospitality of Research Alliance Center for Mathematical Sciences, Tohoku University, where this work was done.


\begin{thebibliography}{xxxxxx}

 \bibitem[And04]{Andre} Y.~Andr\'{e},
  \textit{Une introduction aux motifs (Motifs purs, motifs mixtes, p\'{e}riodes)},
  Panoramas et Synth\`{e}ses, vol. 17,
  Soci\'{e}t\'{e} Math\'{e}matique de France, Paris (2004).
  
 \bibitem[Art86]{Artin} M.~Artin,
 \textit{N\'{e}ron models}, in \textit{Arithmetic geometry}, 
 eds G.~Cornell and J.~H.~Silverman, Springer (1986), 213-230.
 
  \bibitem[B\u{a}d01]{Badescu} L.~B\u{a}descu, 
  \textit{Algebraic surfaces}, Universitext, Springer (2001).
   
  \bibitem[BKL76]{Bloch and Kas and Lieberman} S.~Bloch, A.~Kas, and D.~Lieberman,
  \textit{Zero cycles on surfaces with $p_{g} = 0$}, 
  Compositio Math. {\bf 33} (1976), 135-145.
  
  \bibitem[BL23]{BL} M.~Bolognesi and R.~Laterveer, 
  \textit{A 9-dimensional family of K3 surfaces with finite dimensional motive}, (2023)
  arXiv:2310.11981
  
  \bibitem[BM77]{Bombieri and Mumford II} E.~Bombieri and D.~Mumford, 
  \textit{Enriques' classification of surfaces in Char. p: II}, 
  in \textit{Complex analysis and algebraic geometry}, Iwanami Shoten, Tokyo (1977), 23–42.

 \bibitem[Coo92]{Coombes} K.~Coombes,
  \textit{The K-cohomology of Enriques surfaces}, Contemp. Math. {\bf 126} (1992), 47-57.

\bibitem[CH00]{Corti and Hanamura} A.~Corti and M.~Hanamura,
  \textit{Motivic decomposition and intersection Chow group I},
  Duke Math. J. vol.103 (2000), 459-522.

   \bibitem[CD89]{Cossec and Dolgachev} F.~Cossec and I.~Dolgachev,
   \textit{Enriques surfaces I}, Birkh\"auser (1989).
   
   \bibitem[CDLK]{Cossec and Dolgachev and Liedtke and Kondo}
  F.~Cossec, I.~Dolgachev, C.~Liedtke, with Appendix by S.~Kond\=o,
  \textit{Enriques surfaces I}, New edition, Preprint.

\bibitem[DM91]{DM} C.~Deninger and J. P.~Murre,
\textit{Motivic decomposition of abelian schemes and the Fourier transform}, J. Reine Angew. Math. {\bf 422} (1991), 201-219.

\bibitem[FV20]{FV} L.~Fu and C.~Vial, \textit{Distinguished cycles on varieties with motive of abelian type and the section property},
J. Algebraic Geometry 29 (2020), 53–107.

\bibitem[FV21]{FV21} L.~Fu and C.~Vial, \textit{A motivic global Torelli theorem for isogenous K3 surfaces}, Advances in Mathematics 383 (2021).
  
  \bibitem[Ful84]{Fulton} W.~Fulton,
  \textit{Intersection theory},
  Ergebnisse der Mathematik und ihrer Grenzgebiete. 3. Folge, vol.2, second edition,
  Springer, Berlin (1998).

 \bibitem[GS06]{Gille and Szamuely} P.~Gille and T.~Szamuely, 
   \textit{Central simple algebras and Galois cohomology}, 
   Cambridge Studies in Advanced Mathematics, vol.101,
   Cambridge University Press, Cambridge (2006).
  
\bibitem[G62V]{GrV} A.~Grothendieck, \textit{Technique de descente et th\'eor\`emes d'existence
en g\'eom\'etrie alg\'ebrique. V. Les sch\'emas de Picard: th\'eor\`emes d'existence},
Séminaire N. Bourbaki, 1962, exp. no 232, 143-161.

\bibitem[G62VI]{GrVI} A.~Grothendieck, \textit{Technique de descente et th\'eor\`emes d'existence
en g\'eom\'etrie alg\'ebrique. VI. Les sch\'emas de Picard: propri\'et\'es g\'en\'erales},
Séminaire N. Bourbaki, 1962, exp. no 236, 221-243.
  
  \bibitem[GP02]{Guletskii and Pedrini} V.~Guletskii and C.~Pedrini,
   \textit{The Chow motive of the Godeaux surface}, 
   in \textit{Algebraic Geometry, a Volume in Memory of Paolo Francia}, 
   Walter de Gruyter (2002), 180-195.
    
  \bibitem[Har77]{Hartshorne} R.~Hartshorne,
  \textit{Algebraic Geometry},
  Graduate Texts in Mathematics, vol. 52,
  Springer, New York (1977).
  
  \bibitem[Huy16]{HuyK3} D.~Huybrechts, \textit{Lectures on K3 surfaces}, Cambridge University Press, Cambridge 2016.
  
   \bibitem[Huy19]{Huymotives} D.~Huybrechts,
   \textit{Motives of isogenous K3 surfaces},
   Comment. Math. Helvetici, vol. 94 (2019), no.3, 445-458
  
  \bibitem[ILP09]{ILP} C. Ingalls, A. Logan and O. Patashnick,
  \textit{Explicit coverings of families of elliptic surfaces by squares of curves}, arXiv:2009.07807
  
\bibitem[JY17]{JY}
 Z.~Jiang and Q.~Yin, \textit{On the Chow ring of certain rational cohomology tori},
C.~R.~Acad. Sci. Paris, Ser.~I 355 (2017), 571-576.

\bibitem[KMP07]{Kahn and Murre and Pedrini} B.~Kahn, J.~P.~Murre, and C.~Pedrini,
  \textit{On the transcendental part of the motive of a surface}, 
in \textit{Algebraic cycles and Motives Vol II}, 
London Mathematical Society Lectures Notes Series, vol. 344,
Cambridge University Press, Cambridge (2008), 143-202. 

\bibitem[Kaw22]{Kawabe} D.~Kawabe, \textit{Chow motives of quasi elliptic surfaces}, (2022)
  arXiv:2201.06152, to appear in Osaka Journal of Math.

  
  \bibitem[Kim05]{Kimura} S.~Kimura,
  \textit{Chow groups are finite-dimensional, in some sense}, 
  Math. Ann. {\bf 331} (2005), 173-205.
  
  \bibitem[Kle05]{Kleiman} S.~Kleiman, 
 \textit{The Picard scheme},
 in \textit{Fundamental algebraic geometry : Grothendieck's FGA explained}, 
 Mathematical Surveys and Monographs, vol.123,
 American Mathematical Society (2005).
 
 \bibitem[Lag59]{Lang} S.~Lang, 
 \textit{Abelian varieties}, Interscience, New York (1959).

 \bibitem[LT58]{Lang and Tate} S.~Lang and J.~Tate,
 \textit{Principal homogeneous spaces over abelian varieties},
 Amer. J. Math. {\bf 80} (1958), 659-684.
 
  \bibitem[Lat18]{Late} R. Laterveer, \textit{A family of cubic fourfolds with finite-dimensional motive}, Journal of the Mathematical Society of Japan, vol. 70 (2018), no. 4, 1453-1473.
 
 \bibitem[Liu02]{Liu} Q.~Liu,
 \textit{Algebraic geometry and arithmetic curves}, Oxford University Press (2002).
 
 \bibitem[LLR04]{Liu and Lorenzini and Raynaud} Q.~Liu, D.~Lorenzini, and M.~Raynaud,
 \textit{N\'{e}ron models, Lie algebras, and reduction of curves of genus one},
  Invent. Math. {\bf 157} (2004), 455–518. [Corrigendum, Invent. Math. \textbf{214}
(2018), 593–604].
  
   \bibitem[LT16]{Liu and Tong} Q.~Liu and J. Tong,
   \textit{N\'eron models of algebraic curves},
   Trans. Amer. Math. Soc. 368 (2016), no. 10, 7019–7043.
 
 \bibitem[Man68]{Manin} Y.~I.~Manin,
\textit{Correspondences, motives and monoidal transformations},
Math. USSR. Sbornik {\bf 6} (1968), 439-470.

\bibitem[Mil80]{Milne E} J.~S.~Milne,
 \textit{\'{E}tale cohomology}, 
 Princeton University Press, Princeton, NJ (1980).
 
 \bibitem[Mil86]{Milne J} J.~S.~Milne,
 \textit{Jacobian varieties}, in \textit{Arithmetic geometry},
  eds G.~Cornell and J.~H.~Silverman, Springer (1986) 167-212.
  
  \bibitem[MI21]{MI} M.~Miyanishi and H.~Ito,
   \textit{Algebraic surfaces in positive characteristics: purely inseparable phenomena in curves and surfaces}, World Scientific (2021).
 
 
 \bibitem[Mum70]{Mumford A} D.~Mumford,
 \textit{Abelian varieties},
 Tata Institute of Fundamental Research Studies in Mathematics, Oxford University Press (1970).
 
 \bibitem[Mur64]{Murre On contravariant} J.~P.~Murre,
 \textit{On contravariant functors from the category of preschemes over a field into 
 the category of abelian groups (with an application to the Picard functor)},
 Publ. Math. IHES. {\bf 23} (1964), 5-43.
 
 \bibitem[Mur90]{Murre On the Motive} J.~P.~Murre,
 \textit{On the motive of an algebraic surface},
 J. Reine angew. Math. {\bf 409} (1990), 190–204.
 
  \bibitem[Mur93]{Murre On a} J. P. Murre.
  \textit{On a conjectural filtration on the Chow groups of an algebraic variety. Part I: The general conjectures and some examples}, Indag. Math. (N.S.), 4(2):177–188, 1993.

\bibitem[MNP13]{Murre and  Nagel and Peters} J.~P.~Murre, J.~Nagel, and C.~Peters,
 \textit{Lectures on the theory of pure motives},
 University Lecture Series, vol. 61, American Mathematical Society (2013).
 
 \bibitem[Par88]{Par} K. Paranjape, \textit{Abelian varieties associated to certain K3 surfaces}, Compositio Math. vol. 68 (1988), no.1, 11-22.


 \bibitem[Ped12]{Pedrini 12} C.~Pedrini,
 \textit{On the finite dimensionality of a K3 surface},
 Manuscripta Math. {\bf 138} (2012), 59-72.
 
 \bibitem[PW15]{PW} C. Pedrini and C. Weibel, \textit{Some surfaces of general type for which Bloch's conjecture holds},  in: Period Domains, Algebraic Cycles, and Arithmetic, Cambridge Univ. Press, (2015).
 

 \bibitem[Sch10]{Schroer2010} S.~Schr\"oer,
 \textit{On fibrations whose geometric fibers are nonreduced},
  Nagoya Math. J. {\bf 200} (2010), 35–57.

\bibitem[Ser60]{Serre} J.-P.~Serre, \textit{Morphismes universels et vari\'et\'e d'Albanese},
S\'eminaire Claude Chevalley, tome 4 (1958-1959), exp. no. 10, 1–22.
  
 \bibitem[Sil86]{Silverman} J.~Silverman,
 \textit{The arithmetic of elliptic curves},
 Graduate Texts in Mathematics, vol. 106, second edition, Springer (2009).
 
  \bibitem[Tat52]{Tate} J.~Tate, 
  \textit{Genus change in inseparable extensions of function fields},
   Proc. Amer. Math. Soc. {\bf 3} (1952), 400–406.
   
   \bibitem[Via17]{V} C. Vial, \textit{Remarks on motives of abelian type}, Tohoku Math. J. 69 (2017), 195-220.
 
 \bibitem[Voi03]{Voisin} C.~Voisin,
 \textit{Hodge theory and complex algebraic geometry II}, 
 Cambridge Studies in Advanced Mathematics, vol. 77, 
 Cambridge University Press, Cambridge (2003).
 \end{thebibliography}
 \end{document}